\documentclass[12pt,oneside,reqno]{amsart}
\usepackage[utf8]{inputenc}
\usepackage{mathrsfs,amssymb,multirow,setspace}
\usepackage{mathrsfs,amssymb,multirow,setspace,enumerate}
\usepackage{amsmath,amssymb}
\usepackage{relsize}
\usepackage{xcolor}
\usepackage[a4paper, total={7in, 9in}]{geometry}
\theoremstyle{plain}
\usepackage{graphicx}
\newtheorem{theorem}{Theorem}[section]
\newtheorem{lemma}[theorem]{Lemma}
\newtheorem{proposition}[theorem]{Proposition}

\newtheorem{claim}[theorem]{Claim}
\newtheorem{corollary}[theorem]{Corollary}

\theoremstyle{definition}

\theoremstyle{remark}
\newtheorem{remark}[theorem]{Remark}

\input xy
\xyoption{all}

\def\c{\Delta(B_n)}

\begin{document}
\title[On the Commuting graphs of  Brandt Semigroups]{On the Commuting graphs of Brandt Semigroups}
\author[Jitender Kumar, Sandeep Dalal, Pranav Pandey]{Jitender Kumar, Sandeep Dalal, Pranav Pandey}
\address{Department of Mathematics, Birla Institute of Technology and Science Pilani, Pilani, India}
\email{jitenderarora09@gmail.com,deepdalal10@gmail.com,pranavpandey03061996@gmail.com}

\begin{abstract}
The commuting graph of a finite non-commutative semigroup $S$, denoted by $\Delta(S)$, is the simple graph whose vertices are the non-central elements of $S$ and two distinct vertices $x, y$ are adjacent if $xy = yx$. In the present paper, we study various graph theoretic properties of the commuting graph $\c$ of Brandt semigroup $B_n$ including its diameter, clique number, chromatic number, independence number, strong metric dimension and dominance number. Moreover, we obtain the automorphism group ${\rm Aut}(\c)$ and the endomorphism monoid End$(\c)$ of $\c$. We show that ${\rm Aut}(\c) \cong S_n \times \mathbb{Z}_2$, where $S_n$ is the symmetric group of degree $n$ and $\mathbb{Z}_2$ is the additive group of integers modulo $2$. Further, for $n \geq 4$, we prove that End$(\c) = $Aut$(\c)$. In order to provide an answer to the question posed in \cite{a.Araujo2011}, we ascertained  a class of inverse semigroups whose commuting graph is Hamiltonian.
\end{abstract}

\subjclass[2010]{05C25}

\keywords{Commuting graph, Brandt semigroups, Graphs, Automorphism group of a graph}

\maketitle

\section{Introduction}
The commuting graph of a finite non-abelian group $G$ is a simple graph (undirected graph with no loops or repeated edges) whose vertices are the  non-central elements of $G$ and two distinct vertices $x, y$ are adjacent if $xy = yx$. Commuting graphs of various  groups have been studied by several authors (cf. \cite{a.bates2003commuting,a.bates2003,a.bundy2006connectivity,a.Iranmanesh2008commuting}). Moreover, \cite{a.Segev1999,a.segev2001commuting,a.segev2002anisotropic} use combinatorial parameters of certain commuting graphs to establish long standing conjectures in the theory of division algebras. The concept of commuting graph can be defined analogously for semigroups. Let $S$ be a finite non-commutative semigroup with centre $Z(S) = \{a \in S: ab = ba \; \text{for all} \; b \in S\}$. The commuting graph of  $S$, denoted by $\Delta(S)$, is the simple graph  whose vertex set is $S - Z(S)$ and two distinct vertices $a, b$ are adjacent if $ab = ba$. In 2011, Ara$\acute{\text{u}}$jo et al. \cite{a.Araujo2011} initiated the study of commuting graph on finite semigroups and calculated the diameter of commuting  graphs of various ideals of full transformation semigroup. Also, for every natural number $n \geq 2$, a finite semigroup whose commuting graph has diameter $n$ has been constructed in \cite{a.Araujo2011}. Further, various graph theoretic properties (viz. clique number and diameter) of $\Delta(\mathcal{I}(X))$, where $\mathcal{I}(X)$ is the symmetric inverse semigroup of partial injective transformations on a finite set $X$, have been studied in \cite{a.Araujo2015}.  In order to provide answers to few of the problems posed in \cite{a.Araujo2011}, T. Bauer et al. \cite{a.Bauer2016} have established a semigroup whose knit degree is $3$. For a wider class of semigroups, it was shown in \cite{a.Bauer2016}, that the diameter of their commuting graphs is effectively bounded by the rank of the semigroups. Further, the construction of monomial semigroups with a bounded number of generators, whose commuting graphs have an arbitrary clique number have been provided in \cite{a.Bauer2016}. Motivated with the work in \cite{a.Araujo2011} and the questions posed in its Section 6, in this paper,  we study various graph invariants of the commuting graph associated with an important class of inverse semigroups. This work leads to answer partially to some of the problems posed in \cite{a.Araujo2011}. Moreover, the results obtained in this paper may be useful into the study of commuting graphs on completely $0$-simple inverse semigroups.


  Let $G$ be a finite group. For a natural number $n$, we write $[n] = \{1, 2, \ldots, n\}$. Recall that the \emph{Brandt semigroup}, denoted by $B_n(G)$, has underlying set $([n]\times G \times[n])\cup \{0\}$ and the binary operation  \lq$\cdot$\rq $\;$  on $B_n(G)$ is defined as
 \[ (i,a,j) \cdot (k,b,l) =
                 \left\{\begin{array}{cl}
                 (i,ab, l) & \text {if $j = k$;}  \\
                 0     & \text {if $j \neq k $}
                   \end{array}\right.  \]
 and, for all $\alpha \in B_n (G)$, $\alpha \cdot 0 = 0 \cdot \alpha = 0$. Note that $0$ is the (two sided) zero element in $B_n(G)$. 
 
 \begin{theorem}[{\cite[Theorem 5.1.8]{b.Howie}}]\label{0-simple}
 A finite semigroup $S$ is both completely $0$-simple and an inverse semigroup if and only if $S$ is isomorphic to the semigroup $B_n(G)$ for some group $G$.
 \end{theorem}

 Since all completely $0$-simple inverse semigroups are exhausted by Brandt semigroups, their consideration seems interesting and useful in various aspects. Brandt semigroups  have been studied extensively by various authors, see \cite{a.Volkov2009,a.Morita,a.Sadar2009} and the references therein. When $G$ is the  trivial group, the Brandt semigroup $B_{n}(\{e\})$ is denoted by $B_n$. Thus, the semigroup $B_n$ can be described as the set $([n]\times[n])\cup \{ 0\}$, where $0$ is the zero element and the product $(i, j)\cdot (k, l) = (i, l)$, if $j = k $ and $0$, otherwise. Since Green's $\mathcal{H}$-class of $B_n$ is trivial, it is also known as aperiodic Brandt semigroup. As a Rees matrix semigroup \cite{b.Howie}, $B_n$ is isomorphic to the Rees matrix semigroup $M^0(\{1,\dots, n\}, 1, \{1,\dots, n\}, I_n)$, where $I_n$ is the $n \times n$ identity matrix. Brandt semigroup $B_n$ play an important role in inverse semigroup theory and arises in  number of different ways,  see \cite{a.ciric2000,a.Kamilla2006} and the references therein.  Endomorphism seminear-rings on $B_n$ have been classified  by Gilbert and Samman \cite{a.Gilbert2010}. Further, various aspects of affine near-semirings generated by affine maps on $B_n$ have been studied in \cite{t.jitenderthesis}. The combinatorial study of $B_n$ have been related with theory of matroids and simplicial complexes in \cite{a.Margolis2018}. Various ranks  of $B_n$  have been obtained in \cite{a.Howie1999,a.howie2000rank,b.Mitchell2004}, where some of the ranks of $B_n$ were obtained by using graph theoretic properties of some graph associated  on $B_n$.  Cayley graphs associated with Brandt semigroups have been studied in \cite{a.HaoCayleyBrandt2011, a.KhosraviCayleygraphs2012}.

In this paper, we have investigated various graph theoretic properties of the commuting graph of $B_n$. The paper is arranged as follows. In Section 2, we provide necessary background material and notations used throughout the paper. In Section 3, various graph invariants, namely: diameter, independence number, girth, clique number, chromatic number and vertex connectivity, of $\c$ are obtained. Also, we have shown that $\c$ is Hamiltonian but it is neither planar nor Eulerian. In Section 4, the automorphism group as well as endomorphism monoid of $\c$ is described.

\section{Preliminaries}
In this section, we recall  necessary definitions, results and notations of graph theory from \cite{b.West}.
A graph $\mathcal{G}$ is a pair  $ \mathcal{G} = (V, E)$, where $V = V(\mathcal{G})$ and $E = E(\mathcal{G})$ are the set of vertices and edges of $\mathcal{G}$, respectively. We say that two different vertices $a, b$ are $\mathit{adjacent}$, denoted by $a \sim b$, if there is an edge between $a$ and $b$. We are considering simple graphs, i.e. undirected graphs with no loops  or repeated edges. If $a$ and $b$ are not adjacent, then we write $a \nsim b$. The \emph{neighbourhood} $ N(x) $ of a vertex $x$ is the set all vertices adjacent to $x$ in $ \mathcal G $. Additionally, we denote $N[x] = N(x) \cup \{x\}$. A subgraph  of a graph $\mathcal{G}$ is a graph $\mathcal{G}'$ such that $V(\mathcal{G}') \subseteq V(\mathcal{G})$ and $E(\mathcal{G}') \subseteq E(\mathcal{G})$. A \emph{walk} $\lambda$ in $\mathcal{G}$ from the vertex $u$ to the vertex $w$ is a sequence of  vertices $u = v_1, v_2,\cdots, v_{m} = w$ $(m > 1)$ such that $v_i \sim v_{i + 1}$ for every $i \in \{1, 2, \ldots, m-1\}$. If no edge is repeated in $\lambda$, then it is called a \emph{trail} in $\mathcal{G}$. A trail whose initial and end vertices are identical is called a \emph{closed trail}. A walk is said to be a \emph{path} if no vertex is repeated.  The length of a path is the number of edges it contains. If $U \subseteq V(\mathcal{G})$, then the  subgraph of $\mathcal{G}$ induced by  $U$ is the graph $\mathcal{G}'$ with vertex set $U$, and with two vertices adjacent in $\mathcal{G}'$ if and only if they are adjacent in $\mathcal{G}$. A graph  $\mathcal{G}$ is said to be \emph{connected} if there is a path between every pair of vertex. A graph $\mathcal{G}$ is said to be \emph{complete} if any two distinct vertices are adjacent.  A path that begins and ends on the same vertex is called a \emph{cycle}. A cycle in a graph $\mathcal{G}$ that includes every vertex of $\mathcal{G}$ is called a \emph{Hamiltonian cycle} of $\mathcal{G}$. If $\mathcal{G}$ contains a Hamiltonian cycle, then $\mathcal{G}$ is called a \emph{Hamiltonian graph}.

Also, recall that the \emph{girth} of a graph $\mathcal{G}$ is the length of the shortest cycle in $\mathcal{G}$, if $\mathcal{G}$ has a cycle; otherwise we say the \emph{girth} of $\mathcal{G}$ is $\infty$. The \emph{distance} between vertices $u$ and $w$, denoted by $d(u, w)$, is the length of a minimal path from $u$ to $w$. If there is no path from $u$ to $w$, we say that the distance between $u$ and $w$ is $\infty$. The \emph{diameter} of a connected  graph $\mathcal{G}$ is the maximum distance between two vertices and it is denoted by \emph{diam}($\mathcal G$).  The \emph{degree} of a vertex $v$ is the number of edges incident to $v$ and it is denoted as deg$(v)$. The smallest degree among the vertices of $\mathcal{G}$ is called the \emph{minimum degree} of $\mathcal{G}$ and it is denoted by $\delta(\mathcal{G})$. The \emph{chromatic number} $\chi(\mathcal{G})$ of a graph $\mathcal{G}$ is the smallest positive integer $k$ such that the vertices of $\mathcal{G}$ can be colored in $k$ colors so that no two adjacent vertices share the same color.  A graph $\mathcal{G}$ is \emph{Eulerian} if $\mathcal G$ is both connected and has a closed trail (walk with no repeated edge) containing all the edges of a graph.

\begin{theorem}[{{\cite[Theorem 1.2.26]{b.West}}}]\label{Eulerian}
A connected graph is Eulerian if and only if its every vertex is of even degree.
\end{theorem}


A \emph{clique} of a graph $\mathcal{G}$ is a complete subgraph of $\mathcal{G}$ and the number of vertices in a  clique of maximum size is called the \emph{clique number} of $\mathcal{G}$ and it is denoted by $\omega({\mathcal{G}})$. The graph $\mathcal{G}$ is \emph{perfect} if $\omega(\mathcal{G}') = \chi(\mathcal{G}')$ for every induced subgraph $\mathcal{G}'$ of $\mathcal{G}$. An \emph{independent set} of a graph $\mathcal{G}$ is a subset  of $V(\mathcal{G})$ such that no two vertices in the subset are adjacent in $\mathcal{G}$. The \emph{independence number} of  $\mathcal{G}$ is the maximum size of an independent  set, it is denoted by $\alpha(\mathcal{G})$. A graph $\mathcal{G}$ is \emph{bipartite} if $V(\mathcal{G})$ is the union of two disjoint independent sets. By {\cite[Theorem 1.2.18]{b.West}}, graph $\mathcal{G}$ is bipartite if and only if it does not contain an odd length cycle. Also, recall that a dominating set $D$ of a graph $\mathcal{G}$ is a subset of the vertex set such that every vertex not in $D$ is adjacent to some vertex in $D$ and the number of vertices in a smallest dominating set of $\mathcal{G}$ is called the \emph{dominance number} of $\mathcal{G}$.  A \emph{planar graph} is a graph that can be embedded in the plane, i.e. it can be drawn on the plane in such a way that its edges intersect only at their endpoints.


A \emph{vertex (edge) cut-set} in a connected graph $\mathcal{G}$ is a set of vertices (edges) whose deletion increases the number of connected components of $\mathcal{G}$. The \emph{vertex connectivity} (\emph{edge connectivity}) of a connected graph $\mathcal{G}$ is the minimum size of a vertex (edge) cut-set and it is denoted by $\kappa(\mathcal{G})$ ($\kappa'(\mathcal{G})$). For $k \ge 1$, graph $\mathcal{G}$ is \emph{$k$-connected} if $\kappa(\mathcal{G}) \ge k$. It is well known that $\kappa(\mathcal{G}) \le \kappa'(\mathcal{G}) \le \delta(\mathcal{G})$. An \emph{edge cover} in a graph $\mathcal{G}$ without isolated vertices is a set $L$ of edges such that every vertex of $\mathcal{G}$ is incident to some edge of $L$. The minimum cardinality of an edge cover in $\mathcal{G}$ is called the \emph{edge covering number}, it is denoted by $\alpha'(\mathcal{G})$. A \emph{vertex cover} of a graph $\mathcal{G}$ is a set $Q$ of vertices such that it contains at least one endpoint of every edge of $\mathcal{G}$. The minimum cardinality of a vertex cover in $\mathcal{G}$ is called the \emph{vertex covering number}, it is denoted by $\alpha(\mathcal{G})$. A \emph{matching} in a graph $\mathcal{G}$ is a set of edges with no share endpoints and the maximum cardinality of a matching is called the \emph{matching number} and it is denoted by $\alpha'(\mathcal{G})$. We have the following equalities involving the above parameters.
\begin{lemma}\label{vertex-covering-matching number}
	Consider a graph $\mathcal G$.
	\begin{enumerate}
		\item[(i)] $\alpha(\mathcal G) + \beta(\mathcal G) = |V(\mathcal G)|$.
		\item[(ii)] If $\mathcal G$ has no isolated vertices,  $\alpha'(\mathcal G) + \beta'(\mathcal G) = |V(\mathcal G)|$.
	\end{enumerate}
\end{lemma}

For  vertices $u$ and $v$ in a graph $\mathcal G$, we say that $z$ \emph{strongly resolves} $u$ and $v$ if there exists a shortest path from $z$ to $u$ containing $v$, or a shortest path from $z$ to $v$ containing $u$. A subset $U$ of $V(\mathcal G)$ is a \emph{strong resolving set} of $\mathcal G$ if every pair of vertices of $\mathcal G$ is strongly resolved by some vertex of $U$. The least cardinality of a strong resolving set of $\mathcal G$ is called the \emph{strong metric dimension} of $\mathcal G$ and is denoted by $\operatorname{sdim}(\mathcal G)$. For  vertices $u$ and $v$ in a graph $\mathcal G$, we write $u\equiv v$ if $N[u] = N[v]$. Notice that that $\equiv$ is an equivalence relation on $V(\mathcal G)$.
We denote by $\widehat{v}$ the $\equiv$-class containing a vertex $v$ of $\mathcal G$.
Consider a graph $\widehat{\mathcal G}$ whose vertex set is the set of all $\equiv$-classes, and vertices $\widehat{u}$ and  $\widehat{v}$ are adjacent if $u$ and $v$ are adjacent in $\mathcal G$. This graph is well-defined because in $\mathcal G$, $w \sim v$ for all $w \in \widehat{u}$ if and only if $u \sim v$.  We observe that $\widehat{\mathcal G}$ is isomorphic to the subgraph $\mathcal{R}_{\mathcal G}$ of $\mathcal G$ induced by a set of vertices consisting of exactly one element from each $\equiv$-class. Subsequently, we have the following result of \cite{a.strongmetricdim2018} with $\omega(\mathcal{R}_{\mathcal G})$ replaced by $\omega(\widehat{\mathcal G})$.

\begin{theorem}[{\cite[Theorem 2.2 ]{a.strongmetricdim2018}}]\label{strong-metric-dim}
	Let $\mathcal G$ be a graph with diameter $2$. Then  sdim$(\mathcal G) = |V(\mathcal G)| - \omega(\widehat{\mathcal G})$.
\end{theorem}

When $\mathcal G = \c$, we denote $\widehat{\mathcal G}$ by $\widehat{\Delta}(B_n)$.

The \emph{commuting  graph} of a finite semigroup $S$, denoted by $\Delta(S)$, is the simple graph whose vertices are the non-central elements of $S$ and two distinct vertices $x, y$ are adjacent if $xy = yx$. 

\section{Graph invariants of $\c$}
In this section, we obtained various graph invariants of $\c$. We begin with the results concerning the neighbors (degree) of all  the vertices of $\c$. 
\begin{lemma}\label{nbd}
	In the graph $\c$, we have the following:
	\begin{enumerate}
		\item[\rm (i)] N $[(i, i)] = \{(j, k) \; : \; j,  k \in [n], \; j, k \neq i \} \cup \{(i, i)\}$.
		\item[\rm (ii)]  N $[(i, j)] = \{(i, l) \; : \; l \in [n], l \neq i, j \} \cup \{(l, j) \; : \; l \in [n], l \neq i, j \} \cup \{(k, l) \; : \; k, l \in [n], \; k \neq i, j \; \text{and} \; l \neq i, j  \} \cup \{(i, j)\}$, where $i \neq j$.
	\end{enumerate}
\end{lemma}

\begin{proof}
The result is straightforward for $n = 2$. Now, let $n \ge 3$. Note that $(k, l) \sim (i, i)$ in $\c$ if and only if $k \ne i$ and $l \ne i$.  Now, let $(i, j)$ be a vertex of $\c$, where $i \ne j$. Then vertices $(k, l) \sim (i, j)$ if and only if $(k, l)$ satisfies one of the following:
\begin{enumerate}
\item[(a)]  where $k \ne i, j$ and $l \ne i, j$. 
\item[(b)]  where $k = i$ and  $l \ne i, j$.
\item[(c)]  where $k \ne i, j$ and $l = j$.
\end{enumerate}
\end{proof}

In view of the above proof, we have the following useful remark.
\begin{remark}\label{r.non adjacent vertices}
Two distinct vertices $(i, j)$ and $(k, i)$ are not adjacent in $\c$.
\end{remark}

\begin{corollary}\label{degree of all B_n}
In the commuting graph $\c$,  the degree of idempotent vertices is $(n - 1)^2$ and the degree of non-idempotent vertices is $n(n - 2)$. 
\end{corollary}

\begin{corollary}
The minimum degree of $\c$ is $n(n-2)$.
\end{corollary}

\begin{theorem}\label{Hamiltonian}
	For $n \ge 3$, the commuting graph $\c$ satisfies the following properties:
	\begin{enumerate}
		\item[\rm (i)]  $\Delta(B_n)$ is not Eulerian.
		\item[\rm (ii)] The girth of $\Delta(B_n)$ is $3$.
		\item[\rm (iii)]  $\Delta(B_n)$ is Hamiltonian.
		\item[\rm (iv)]  $\Delta(B_n)$ is connected and ${\rm diam}(\Delta(B_n)) = 2$.
	\end{enumerate}
\end{theorem}

\begin{proof}
\begin{enumerate}
\item[(i)] If $n$ is even, then the ${\rm deg}(i, i) = (n - 1)^2$ is odd. If $n$ is odd, the vertex $(i, j)$, where $i \ne j$, has degree $ n(n-2)$, which is odd. Thus by Theorem \ref{Eulerian}, $\c$ is not Eulerian.

\smallskip

\noindent		
(ii) In order to find the girth of $\c$, note that $(1, 1) \sim (2, 2) \sim (3, 3) \sim (1, 1)$ is a smallest cycle with length $3$. Thus, we have the result.
		
\smallskip

\noindent		
(iii) For $n = 3$, note that $(1, 2) \sim (1, 3) \sim  (2, 3) \sim  (2, 1) \sim (3, 1) \sim  (3, 2) \sim (1, 1) \sim (2, 2)  \sim (3, 3) \sim (1, 2)$ is a Hamiltonian cycle in $\c$. Since the minimum degree of $\c$ is $n(n-2)$ (cf. Lemma \ref{degree of all B_n}), for $n \ge 4$, we have $n(n - 2) - \frac{n^2}{2} = \frac{n}{2}(n - 4) \ge 0$. Hence, by {\cite[Theorem 7.2.8]{b.West}},  $\c$ is Hamiltonian.

\smallskip

\noindent		
(iv) From the Figure 1, note that the ${\rm diam}(\Delta(B_3))$ is $2$. Now for $n \ge 4$, let $(a, b), (c, d) \in B_n$. If all of $a, b, c, d$ are distinct, then $(a, b) \sim (c, d)$. If at most three from $a, b, c, d$ are distinct, then there exists $i \ne a, b, c, d$ such that $(a, b) \sim (i, i) \sim (c, d)$ is a path. Thus, any two vertices of $\c$ are connected by a path with maximum distance two.
	\end{enumerate}
\end{proof}
\begin{figure}[h!]
	\centering
	\includegraphics[width=0.5 \textwidth]{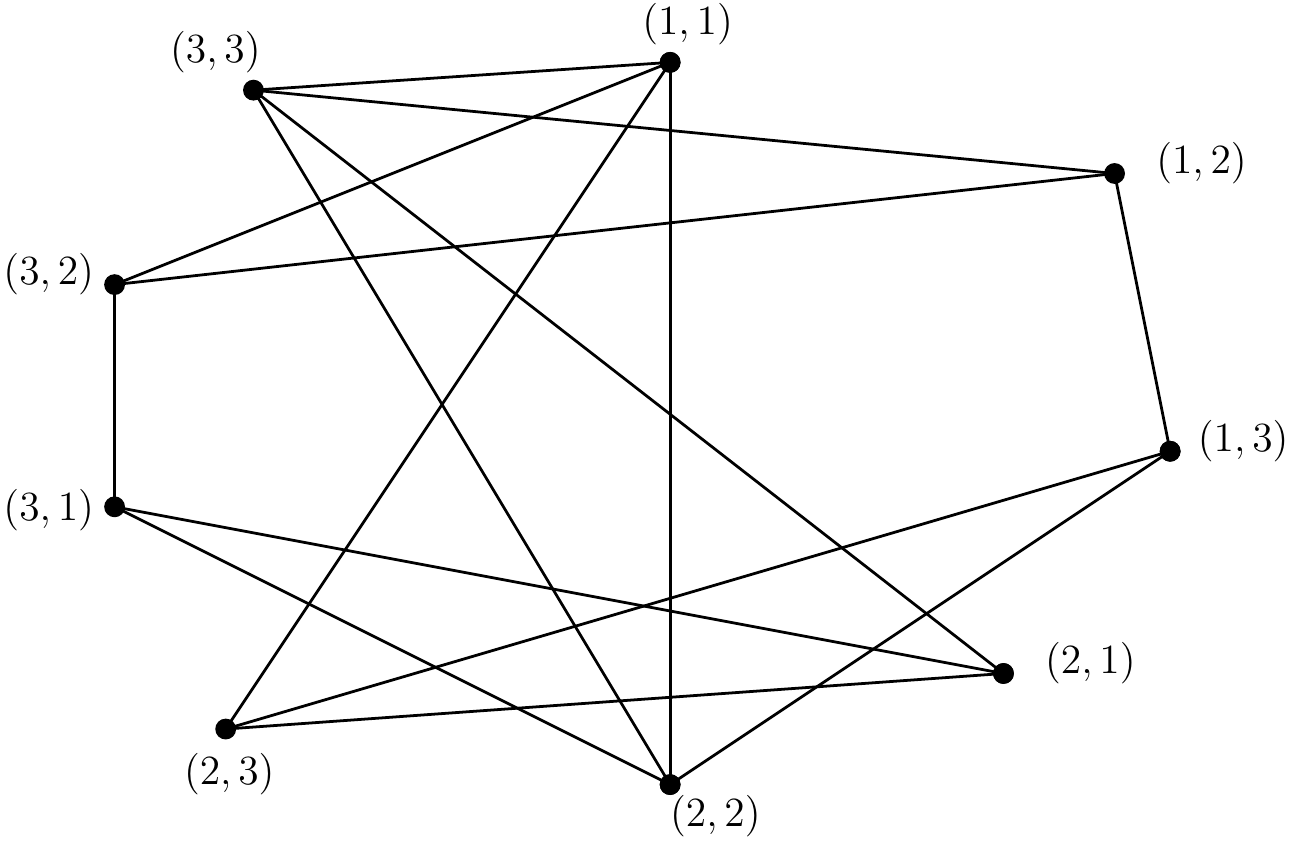}
	\caption{ The commuting graph of $B_3$}\label{B_3}
\end{figure}

Since $\c$ is Hamiltonian so we have the following corollary.

\begin{corollary} For $n \geq 3$, the matching  number of  $\c$ is $\left\lfloor \frac{n^2}{2}\right \rfloor$.
\end{corollary}

In view of Lemma \ref{vertex-covering-matching number}(ii), we have the following corollary.

\begin{corollary}
For $n \geq 3$, the edge  covering number  of $\c$ is  $n^2 - \left\lceil \frac{n^2}{2}\right\rceil$.
\end{corollary}

In the following theorem, we investigate the independence number, dominance number, planarity and perfectness of $\c$.
\begin{theorem}\label{independence number}
For $n \ge 2$, we have
\begin{enumerate}
\item[\rm (i)] the independence number of $\Delta(B_n)$ is $3$.
\item[\rm (ii)] the dominance number of $\Delta(B_n)$ is $3$.
\item[\rm (iii)] $\Delta(B_n)$ is planar if and only if $n = 2$.
\item[\rm (iv)] $\Delta(B_n)$ is perfect if and only if $n = 2$.		
\end{enumerate}
\end{theorem}

\begin{proof}
\begin{enumerate}
\item[(i)] For $n \geq 2$, first note that the set $I' = \{(1, 1), (1, 2), (2, 1)\}$ is an independent set of $\c$. In fact, $I$ is of maximum size for $n = 2$. Thus, the result hold for $n = 2$. Now, to prove the result we show that any independent set in $\c$, where $n \geq 3$, is of size at most $3$. Let $I$ be an independent subset of $\c$. If $I$ does not contain any non-idempotent vertex of $\c$, then clearly $I = \{(i, i)\}$, for some $i \in [n]$. Thus, $|I|\le 3$. We may now suppose
$(i, j)$, where $i \ne j$, belongs to $I$. Note that each of the set $A = \{(i, i), (j, j)\}$, $B = \{(x, i) \; : \; x \ne i\}$ and $C = \{(j, y) \; : \; y \ne j \}$ of vertices forms a complete subgraph of $\c$. For $(k, l) \notin A \cup B \cup C$, we have $k \ne j$ and $l \ne i$. In this case, $(k, l) \sim (i, j)$. Thus, the independent set $I$ must contained in $A \cup B \cup C$. Being an independent set $I$ can contain at most one element from each of these sets. Consequently, $|I| \le 3$. 
		
\smallskip		
\noindent
(ii) By part (i), we have $\alpha(\c) = 3$. Further by \cite[Lemma 3.1.33]{b.West}, the dominance number of $\c$ is at most $3$. Now we prove the result by showing that  any dominating subset of $\c$ contains at least three elements. Let $D$ be a dominating subset of $\c$. In view of Corollary \ref{degree of all B_n}, we do not have a vertex of $\c$ whose degree is $n^2 - 1$ so that $|D| \ne 1$. Suppose $D = \{(a, b), (c, d)\}$. If $a,b, c, d$ all are distinct, then it can be verified that the vertex $(b, c)$ is not adjacent to any element of $D$; a contradiction for $D$ to be a dominating set. If $D$ contains an idempotent, say $(a, b) = (i, i)$, then clearly both $c, d$ can not be equal to $i$. Without loss of generality, let $c \ne i$. Then it is easy to observe that $(i, c)$ is not adjacent to all the elements of $D$; a contradiction. Thus, in this case ($|D| = 2$), the dominating set $D$ can not contain any idempotent vertex and it can not be of the form $\{(a, b), (c,d)\}$, where all $a, b, c, d$ are distinct. Now we have the following remaining cases:
		
\noindent\textbf{Case 1}: $c \in \{a, b\}$. Then the vertex $(d, a)$ is not adjacent to any element of $D$; a contradiction.
		
\noindent\textbf{Case 2}: $d \in \{a, b\}$. Then the vertex $(b, c)$ is not adjacent to any element of $D$; again a contradiction.
		
		
Thus, a dominating set of two elements in $\c$ is not possible. Consequently, $|D| \ge 3$. 
		
\smallskip
\noindent
(iii) For $n = 2$, it is easy to observe that  $\c$ is planar. For the converse part, let $n \geq 3$. It is sufficient to show that some induced subgraph of $\c$ is not planar. From the subgraph of $\c$ in Figure \ref{sub}, if we apply edge contraction on the vertices $(1, 2)$ and $(2, 1)$, then we get a complete bipartite graph $K_{3, 3}$. Hence, by Kuratowski's theorem, $\c$ is not planar.
\begin{figure}[h!]
\centering
\includegraphics[width=0.5 \textwidth]{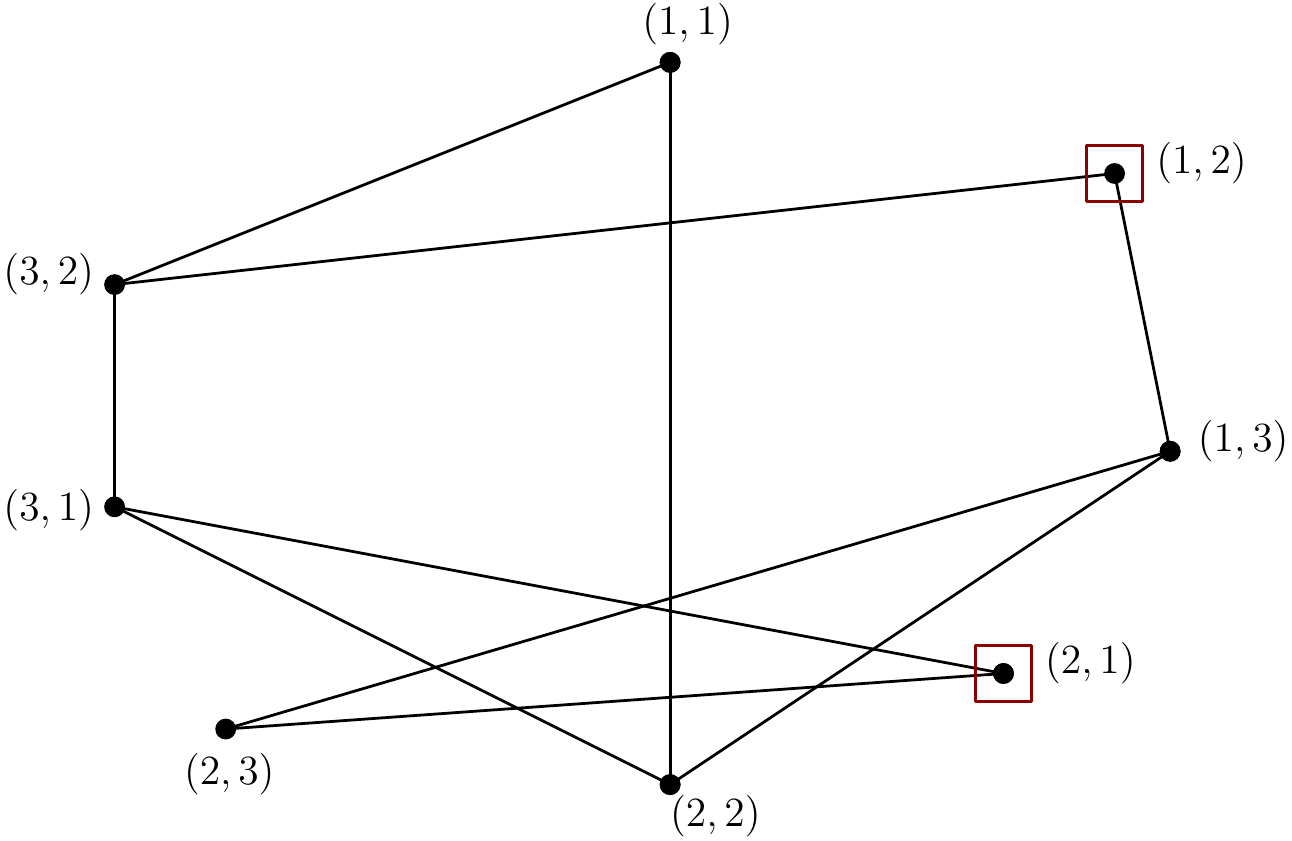}
\caption{The subgraph of  $\Delta(B_n)$}\label{sub}
\end{figure}
		
\smallskip
\noindent
(iv) For $n = 2$, it is easy to verify that  $\c$ is perfect. On the other hand, let $n \geq 3$. In order to prove  that $\c$ is not perfect, we prove that the chromatic number  and clique number of an induced subgraph of $\c$  are not equal. In fact, we show that for the subgraph induced by $U = V(\Delta(B_3)) \setminus \{(3, 3)\}$,  $\omega(\Delta(U)) = 2$ whereas $\chi(\Delta(U)) = 3$. By part (i), since the independence number of $\c$ is $3$, we must have at least three subsets in any chromatic partition of $\Delta(U)$. Thus, $\chi(\Delta(U)) \geq 3$. Further, note that the sets $A_1 = \{(1, 1), (1, 2), (2, 1)\}, \; A_2 = \{(2, 2), (2, 3), (3, 2)\}$ and  $A_3 = \{(1, 3), (3, 1)\}$ forms a chromatic partition of the vertex set of $\Delta(U)$. Hence, $\chi(\Delta(U)) = 3$.
\end{enumerate}
\end{proof}

In view of Lemma \ref{vertex-covering-matching number}, we have the following consequences of Theorem \ref{independence number}(i).
\begin{corollary} 
For $n \geq 2$, the vertex covering number of  $\c$ is  $n^2 - 3$.
\end{corollary}


\textbf{Notation:} We denote $\mathcal K$ as the set of all cliques of $\c$ having no idempotent element and $\mathcal E$ as the set of non-zero idempotents of $B_n$.  

In order to obtain the clique number of $\c$, the following lemma is useful.

\begin{lemma}\label{clique inequality}
	For $K  \in \mathcal K$, we have
	\[ |K| \leq 
	\left\{\begin{array}{cl}
	\frac{n^2}{4} & \text {if $n$ is even;}  \\
	\vspace{-.02cm}\\
	\frac{n^{2} -1}{4} & \text {if $n$ is odd.}
	\end{array}\right. \]
\end{lemma}

\begin{proof}
	Suppose $K$ is of maximum size. Consider $A = \{a \in [n]  :  (a, y) \in K \; \text{for some} \; y \in [n] \}$ and $B = \{b \in [n]  :  (x, b) \in K \; \text{for some} \; x \in [n] \}$.
	If $t \in A \cap B$, then there exist $p, q \in [n]$ such that $(t, p), (q, t) \in K$. Since $K$ is a clique, we get $(t, p) \sim (q,t)$  so that $(t, p)(q, t) = (q, t)(t, p)$. Consequently, $p = q = t$ gives $(t,t) \in K$; a contradiction. Thus, $A$ and $B$ are disjoint subsets of $[n]$ and so $A \times B$ does not contain an idempotent.  
	If $(a,b), (c,d) \in A \times B$, then $a \neq d$ and $b \neq c$. As a consequence, $(a, b) \sim (c, d)$. Thus, $A\times B$ is  a clique such that $K \subseteq A \times B$. Since $K$ is a clique of maximum size which does not contain an idempotent, we get $K = A \times B$. If $|A| = k$, then $|B| = n-k$ because $A \times B$ is a clique of maximum size. Further, $|K| = |A||B| = k(n-k)$. If $n$ is even, note that $|K| = \frac{n^2}{4}$ which attains at $k = \frac{n}{2}$. Otherwise, $|K| = \frac{n^2 -1}{4}$ which attains at either $k = \frac{n-1}{2}$ or $k = \frac{n+1}{2}$.
\end{proof}

In view of the proof of Lemma \ref{clique inequality}, we have the following corollary.

\begin{corollary}\label{clique-non-idempotent}
For $n \geq 4$, there exists $K \in \mathcal K$  such that \[ |K| =  
\left\{\begin{array}{cl}
\frac{n^2}{4} & \text {if $n$ is even;}  \\
\vspace{-.02cm}\\
\frac{n^{2} -1}{4} & \text {if $n$ is odd.}
\end{array}\right. \]
\end{corollary}

\begin{lemma}\label{clique-4}
For $n \in \{2, 3, 4\}$, the set  $\mathcal E$ forms a clique of maximum size. Moreover, in this case $\omega(\c) = n$.
\end{lemma}

\begin{proof}
By Figure \ref{B_3}, note that  $\{(1, 1), (2, 2)\}$ and $\{(1, 1), (2, 2), (3, 3)\}$ forms a clique of maximum size for $n = 2$ and $ 3$, respectively. Now, for $n = 4$, clearly $ K = \{(1, 1), (2, 2), (3, 3), (4, 4)\}$ is a clique in $\Delta(B_4)$. Suppose  $K'$ is a clique of maximum size. If  $K'$ does not contain an idempotent, then by Lemma \ref{clique inequality}, $|K'| = 4$. Thus, $K$ is also a clique of maximum size. On the other hand, we may now assume that $K$ contains an idempotent. Without loss of generality, let $(4, 4) \in K'$. Then $K' \setminus \{(4, 4)\}$ is a clique of maximum size in $\Delta(B_3)$. Since $\{(1, 1), (2, 2), (3, 3)\}$ is the only clique in $\Delta(B_3)$ of maximum size. Thus, $K' \setminus \{(4, 4)\} = \{(1, 1), (2, 2), (3, 3)\}$. Consequently, $K' =  \{(1, 1), (2, 2), (3, 3), (4, 4)\} = K$. Hence, we have the result.
\end{proof}  

From the proof of Lemma \ref{clique inequality} and Lemma \ref{clique-4}, we have the following remark.
\begin{remark}\label{r.clique}
For $n= 4$, let $K$ be any clique in $\c$ of size $4$. Then $K$ is either $\mathcal E$ or $K = A \times B$, where $A$ and $B$ are disjoint subset of $\{1, 2, 3, 4\}$ of size two. 
\end{remark}

\begin{theorem}\label{clique equality}
For $n > 4$, the clique number of $\c$ is given below:
	\[ \omega(\Delta(B_n)) =
	\left\{\begin{array}{cl}
	\frac{n^2}{4} & \text {if $n$ is even;}  \\
	\vspace{.1cm}\\
	\frac{n^{2} -1}{4} & \text {if $n$ is odd}.
	\end{array}\right. \]
\end{theorem}

\begin{proof}
In view of Lemma \ref{clique inequality}, it is sufficient to prove that any clique of maximum size in $\c$ contains only non-idempotent vertices. Suppose $K$ is a clique of maximum size such that $K$ contains $m$ idempotents viz. $(i_1, i_1),(i_2, i_2), \ldots, (i_m, i_m)$. Without loss of generality, we assume that $\{i_1, i_2, \ldots, i_m \} = \{n - m +1, n -m + 2, \ldots, n \}$. For $1 \le r \le m$, $K$ contains $(i_r, i_r)$ and no element of the form $(x, i_r)$ or $(i_r, x)$ $(x \in [n]), x \ne i_r$ is in $K$. Thus $K \setminus \{(i_1, i_1), \ldots,(i_m, i_m) \}$ is a clique in $\Delta(B_{n-m})$ which does not contain any idempotent. Clearly, $|K \setminus \{(i_1, i_1), \ldots,(i_m, i_m) \}| = \omega(\Delta(B_{n-m}))$. Then by Corollary \ref{clique-non-idempotent}
	\[|K \setminus \{(i_1, i_1), \ldots,(i_m, i_m) \}| = \left\{\begin{array}{cl}
	\frac{(n - m)^2}{4} & \text {if $n - m$ is even;}  \\
	\\
	\frac{(n - m )^{2} -1}{4} & \text {if $n - m$ is odd.}
	\end{array}\right.\]
	Thus, \[|K| =  \left\{\begin{array}{cl}
	\frac{(n - m)^2}{4} + m & \text {if $n - m$ is even;}  \\
	\frac{(n - m )^{2} -1}{4} + m & \text {if $n - m$ is odd.}
	\end{array}\right.\]
	
	Since $n >4$ and for $m > 0$, one can observe that
	\[ |K| < \left\{\begin{array}{cl}
	\frac{n^2}{4} & \text {if $n$ is even;}  \\
	\frac{n^{2} -1}{4} & \text {if $n$ is odd;}
	\end{array}\right.\]
	a contradiction of the fact that $K$ is a clique of maximum size (see proof of lemma \ref{clique inequality}).  Thus $K$ has no idempotent.
\end{proof}

By the proof of Lemma \ref{clique inequality} and Theorem \ref{clique equality}, we have the following remarks.

\begin{remark}\label{non-idempotent_clique}
For $n > 4$, let $K$ be a clique  of maximum size in $\c$. Then all elements of $K$ are non-idempotent.
\end{remark}

\begin{remark}\label{non-idempotent_clique-II}
For $n > 4$ and $(i, j) \notin \mathcal E$,  there exists a clique $ K$ of maximum size  such that $(i, j) \in  K$.
\end{remark}

In view of Lemma \ref{nbd}, note that for each vertex $\widehat v$ of $\widehat{\Delta}(B_n)$ we have $\widehat v = \{v\}$. Thus, $\omega(\widehat{\Delta}(B_n)) = \omega(\c)$. Hence by Theorems \ref{strong-metric-dim} and  \ref{clique equality}, we have the following result.

\begin{theorem}
	For $n \geq 2$, we have \[{\rm sdim}(\c) = \left\{\begin{array}{cl}

	\frac{3n^2}{4} & \text {if $n$ is even;}\\
	\\
	\frac{3n^{2} + 1}{4} & \text {if $n$ is odd.}
	\end{array}\right. \]
\end{theorem}


Now we obtain the chromatic number of $\c$. For $n \in \mathbb N$, we write $n = 3a +r$ where $0 \le r \le 2$ and $a \in \mathbb N$. Consider 

\[\mathcal A_{m,x} =\{(m+x,m),(m,n-2m+2-x),(n-2m+2-x,m+x)\}\]
and 
\[\mathcal B_{\ell,y} =\{(\ell,\ell-y),(\ell-y,2n-2\ell+2+y),(2n-2\ell+2+y,\ell)\}.\]

In order to obtain $\chi(\c)$, first we prove the following claims which will be useful to obtain the chromatic partition of $V(\c)$.\\

\begin{claim}\label{c.r=0}
Let $n = 3a$. Then
\begin{enumerate}
\item[\rm (i)] for  $m \in \{1, 2, \ldots, a\}$ and $x \in \{0, 1, 2, \ldots, n-3m + 1\}$,  $\mathcal A_{m,x}$ are the  disjoint independent subsets of $\c$.

\item[\rm (ii)] for $\ell \in \{2a +  1, 2a + 2, \ldots, n\}$ and $y \in \{0, 1, 2, \ldots, 3\ell-2n - 3\}$,  $\mathcal B_{\ell,y}$ are the  disjoint independent subsets of $\c$.
\end{enumerate}
\end{claim}

\begin{proof}
\begin{enumerate}
\item[(i)] 
For  $m \in \{1, 2, \ldots, a\}$ and $x \in \{0, 1, 2, \ldots, n-3m + 1\}$, note  that $m, m + x, n-2m+2-x \in [n]$. Thus, $\mathcal A_{m,x} \subseteq B_n$. By Remark \ref{r.non adjacent vertices}, any pair of vertices in $\mathcal A_{m,x}$ are not adjacent and so each  $\mathcal A_{m,x}$ is an independent subset of $\c$. Now we prove that any two distinct subsets $\mathcal A_{m_1, x_1}$ and $\mathcal A_{m_2, x_2}$ are disjoint. If possible, let $(m_1+x_1,m_1)\in \mathcal A_{m_2, x_2}$. Clearly, $(m_1 + x_1, m_1) \ne (m_2 + x_2, m_2)$. Then either $(m_1+x_1,m_1)= (m_2,n-2m_2+2-x_2)$ or $(m_1+x_1,m_1)=(n-2m_2+2-x_2,m_2+x_2)$. If $(m_1+x_1,m_1)=(m_2,n-2m_2+2-x_2)$, we get $m_1+x_1=m_2$ and $m_1=n-2m_2+2-x_2$. As a consequence, $x_2=(n-3m_2+1)+x_1 + 1\geq n-3m_2+1$; a contradiction of $x_2 \leq n -3m_2 +1$. Similarly, for $(m_1+x_1,m_1)=(n-2m_2+2-x_2,m_2+x_2)$, we get $x_1 = n - 3m_1 + 2 + x_2 > n - 3m_1 + 1$; a contradiction.
Thus, $(m_1+x_1,m_1)\notin \mathcal A_{m_2, x_2}$. Analogously, one can check that $(m_2+x_2,m_2)\notin \mathcal A_{m_1, x_1}$. Now, if $(m_1,n-2m_1+2-x_1)\in \mathcal A_{m_2, x_2}$, then either $(m_1,n-2m_1+2-x_1) = (m_2,n-2m_2+2-x_2)$ or $(m_1,n-2m_1+2-x_1)=(n-2m_2+2-x_2,m_2+x_2)$. For $(m_1,n-2m_1+2-x_1) = (m_2,n-2m_2+2-x_2)$, we obtain $m_1 = m_2$ and $x_1 = x_2$; a contradiction so $(m_1,n-2m_1+2-x_1)=(n-2m_2+2-x_2,m_2+x_2)$. Thus, we have $x_2=n-3m_2+2+x_1$ implies $x_2 > n-3m_2+1$; a contradiction.
 Therefore, $(m_1,n-2m_1+2-x_1)\notin \mathcal A_{m_2, x_2}$. By replacing $m_1, x_1$ with $m_2, x_2$ respectively, we get $(m_2,n-2m_2+2-x_2)\notin \mathcal A_{m_1, x_1}$.  If $(n-2m_1+2-x_1,m_1+x_1)\in \mathcal A_{m_2, x_2}$, then $(n-2m_1+2-x_1,m_1+x_1)=(n-2m_2+2-x_2,m_2+x_2)$. As a consequence, $m_1+x_1=m_2+x_2$ and $2m_1+x_1=2m_2+x_2$ gives $m_1=m_2$ and $x_1=x_2$; a contradiction. Thus, $\mathcal A_{m_1, x_1}\cap \mathcal A_{m_2, x_2}=\varnothing$.

\item[(ii)] For  $\ell \in \{2a +  1, 2a + 2, \ldots, n\}$ and $y \in \{0, 1, 2, \ldots, 3\ell-2n - 3\}$,  note that $\ell, \ell - y, 2n - 2\ell + 2 + y \in [n]$. Thus $\mathcal B_{\ell,y} \subseteq B_n$. By Remark \ref{r.non adjacent vertices}, any pair of vertices in $\mathcal B_{\ell,y}$ are not adjacent and so each  $\mathcal B_{\ell,y}$ is an independent subset of $\c$. Now we prove that any two distinct subsets $\mathcal B_{\ell_1, y_1}$ and $\mathcal B_{\ell_2, y_2}$ are disjoint. If possible, let $(\ell_1,\ell_1 - y_1)\in \mathcal B_{\ell_2, y_2}$. Clearly $(\ell_1,\ell_1 - y_1) \ne (\ell_2,\ell_2 - y_2)$. Then, we get either $(\ell_1, \ell_1-y_1) = (\ell_2- y_2 ,2n -2\ell_2 + 2 + y_2)$ or $(\ell_1, \ell_1 - y_1)=(2n -2\ell_2 + 2 + y_2,\ell_2)$. If $(\ell_1, \ell_1-y_1) = (\ell_2- y_2 ,2n -2\ell_2 + 2 + y_2)$, then $\ell_1 = \ell_2 - y_2$ and $\ell _1 - y_1 = 2n - 2\ell_2 + 2 + y_2$ so $\ell_1 - y_1 = 2n - 2(\ell_1 + y_2) + 2 + y_2$. Therefore,  we get $y_1 = (3\ell_1 - 2n - 3) + y_2 + 1$ which is not possible as $y_1 \le 3\ell_1 -2n -3$. As a result, $(\ell_1, \ell_1 - y_1)=(2n -2\ell_2 + 2 + y_2,\ell_2)$ gives $\ell_1 = 2n -2\ell_2 + 2 + y_2$ and $\ell_1 - y_1 = \ell_2$. Consequently, $\ell_2 + y_1 =  2n -2\ell_2 + 2 + y_2$ implies $y_2 = 3\ell_2 -2n -3 + y_1 + 1$; a contradiction of $y_2 \le 3\ell_2 - 2n - 3$. Thus, $(\ell_1, \ell_1 - y_1) \notin \mathcal B_{\ell_2,y_2}$. Analogously, one can show that  $(\ell_2, \ell_2 - y_2) \notin \mathcal B_{\ell_1,y_1}$. If
$(\ell_1-y_1,2n-2\ell_1 + 2 + y_1) \in \mathcal B_{\ell_2,y_2}$, then either $(\ell_1-y_1,2n-2\ell_1 + 2 + y_1) = (\ell_2-y_2,2n-2\ell_2 + 2 + y_2)$ or $(\ell_1-y_1,2n-2\ell_1 + 2 + y_1) = (2n-2\ell_2+2+y_2,\ell_2)$. Suppose $(\ell_1-y_1,2n-2\ell_1 + 2 + y_1) = (\ell_2-y_2,2n-2\ell_2 + 2 + y_2)$, we obtain $\ell_1 - y_1 = \ell_2 - y_2$ and $2n - 2\ell_1 + 2 + y_1 = 2n - 2\ell_2 + 2 + y_2$. It follows that $2(\ell_2 - \ell_1) = \ell_2 - \ell_1$ which is possible only if $\ell_1 = \ell_2$ and $y_1 = y_2$; a contradiction. Therefore, we get $(\ell_1-y_1,2n-2\ell_1 + 2 + y_1) = (2n-2\ell_2+2+y_2,\ell_2)$ and this implies $\ell_1 - y_1 = 2n - 2(2n - 2\ell_1 +2 + y_1) + 2 + y_2$. As a result, $y_1  = 3\ell_1 - 2n - 3 + y_2 + 1$ which is not possible as $y_1 \le 3\ell_1 -2n -3$. Thus, $(\ell_1-y_1,2n-2\ell_1 + 2 + y_1) \notin \mathcal B_{\ell_2,y_2}$.  In a similar lines one can show that $(\ell_2-y_2,2n-2\ell_2 + 2 + y_2) \notin \mathcal B_{\ell_1,y_1}$. Thus, if $\mathcal B_{\ell_1,y_1} \cap \mathcal B_{\ell_2,y_2} \ne \varnothing $, then we must have $(2n - 2\ell_1 + 2 + y_1, \ell_1 ) = (2n - 2\ell_2 + 2 + y_2, \ell_2)$. It follows that $\ell_1 = \ell_2$ and $y_1  = y_2$; again a contradiction. Hence, the result hold.
\end{enumerate}
\end{proof}

The proof of the following claims is in the similar lines to the proof of the \textbf{Claim \ref{c.r=0}}, hence omitted.

\begin{claim}\label{c.r=1}
Let $n = 3a + 1$. Then
\begin{enumerate}
\item[\rm (i)] for $m \in \{1, 2, \ldots, a\}$ and $x \in \{0, 1, 2, \ldots, n-3m + 1\}$,  $\mathcal A_{m,x}$ are the  disjoint independent subsets of $\c$.
		
\item[\rm (ii)] for $\ell \in \{2a +  2, 2a + 3, \ldots, n\}$ and $y \in \{0, 1, 2, \ldots, 3\ell-2n - 3\}$,  $\mathcal B_{\ell,y}$ are the  disjoint independent subsets of $\c$.
\end{enumerate}
\end{claim}

\begin{claim}\label{c.r=2}
Let $n = 3a + 2$. Then
\begin{enumerate}
\item[\rm (i)] for $m \in \{1, 2, \ldots, a+1\}$ and $x \in \{0, 1, 2, \ldots, n-3m + 1\}$,  $\mathcal A_{m,x}$ are the  disjoint independent subsets of $\c$.
		
\item[\rm (ii)] for  $\ell \in \{2a +  3, 2a + 4, \ldots, n\}$ and $y \in \{0, 1, 2, \ldots, 3\ell-2n - 3\}$,  $\mathcal B_{\ell,y}$ are the  disjoint independent subsets of $\c$.
\end{enumerate}
\end{claim}

In view of above claims, a visual representation of $\mathcal A_{m,x}$ and $\mathcal B_{\ell, y}$ can be observed in the matrix given in Figure \ref{fig-A_mx}. Independent sets $\mathcal A_{1, x}, \mathcal A_{2, x}, \cdots$ covers the vertices through dashed triangles, whereas the independent sets $\mathcal B_{n, y}, \mathcal B_{n-1, y}, \cdots$ covers the vertices of $V(\c)$ on doted triangles as shown in Figure \ref{fig-A_mx}.
\begin{figure}[h!]
	\centering
	\includegraphics[width=0.8 \textwidth]{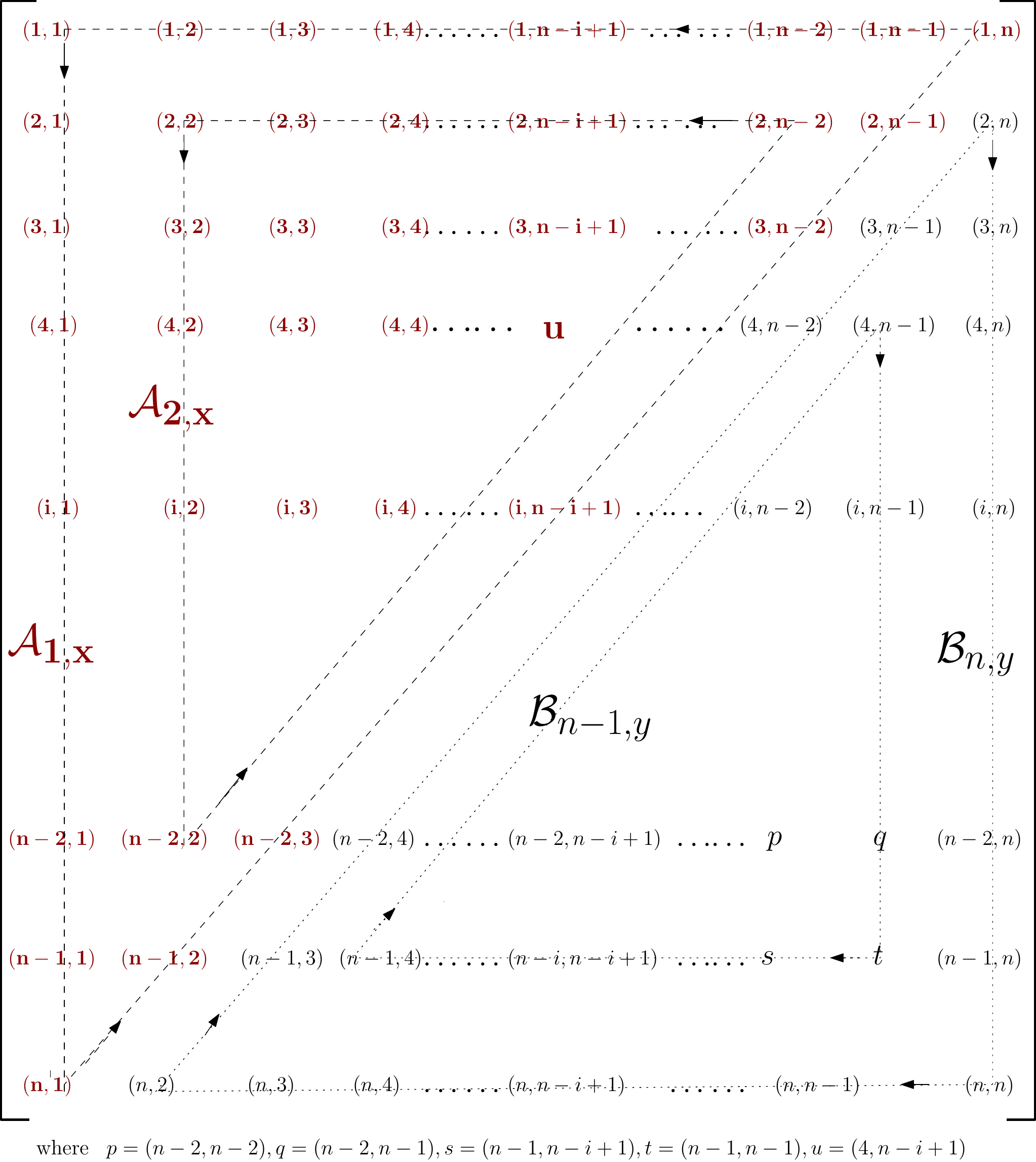}
	\caption{Visual Representation of $\mathcal A_{m,x}$ and $\mathcal B_{\ell, y}$}\label{fig-A_mx}
\end{figure}


\begin{theorem}\label{chromatic number}
For $n \ge 2$, we have $\chi(\c) = \left \lceil \frac{n^2}{3}\right\rceil$.
\end{theorem}
\begin{proof}
For $n = 2$, it is straightforward that $\chi(\c) = 2$. Since  $\chi(\mathcal G)\geq \frac{|V(\mathcal G)|}{\alpha{(\mathcal G})}$  (cf.  {\cite[Proposition 5.1.7]{b.West}}) so that by Lemma \ref{independence number}, we have $\chi (\Delta (B_n))\geq \frac{n^2}{3}$. In order to  obtain the result, we shall provide a partition of $V(\c)$ into $\left \lceil \frac{n^2}{3}\right\rceil$ independent subsets. Now, we have the following cases:\\
\textbf{Case 1:} $n = 3a$. First we prove that the sets $\mathcal A_{m,x}$ and $\mathcal B_{\ell,y}$, where $m, x, \ell, y$ are given in \textbf{Claim \ref{c.r=0}} are disjoint with each other. Note that for $(i,j)\in \mathcal A_{m,x}$ and $(k,t)\in \mathcal B_{\ell,y}$, we have $i+j\leq n+1$ and $k+t\geq n+2$. Thus, $\mathcal A_{m,x}\cap \mathcal B_{\ell,y}=\varnothing$. Now we shall show that $\mathcal A_{m,x}$ and $\mathcal B_{\ell,y}$ forms a partition of $V(\c)$. It is sufficient to show that  $|(\cup \mathcal A_{m,x})\cup(\cup \mathcal B_{\ell,y})|=n^2$. If $m = a, a-1, \ldots, 1$ then $x \in \{0, 1\}, \; x \in \{0,1, 2, 3, 4\}, \ldots, x \in \{0, 1, 2, \ldots, n-2 \}$, respectively. Thus, the total number of  sets of the form $\mathcal A_{m,x}$ is $2 +5 + \cdots + n-1 = \frac{n(n+1)}{6}$. Similarly,  the total number of the sets of the form $\mathcal B_{\ell,y}$ is $1 + 4 + \cdots + n-2 = \frac{n(n - 1)}{6}$.  Consequently, we have  $(\cup \mathcal A_{m,x})\cup(\cup \mathcal B_{\ell,y})=V(\c)$. Thus, we have a partition of $V(\c)$ into $\frac{n^2}{3}$ independent sets. Therefore, $\chi (\Delta (B_n))\leq{\frac{n^2}{3}}$. Hence, $\chi (\Delta (B_n))={\frac{n^2}{3}}$.\\
  \textbf{Case 2:} $n = 3a + 1$. By the similar arguments used in \textbf{Case 1}, the sets $\mathcal A_{m,x}$ and $\mathcal B_{\ell,y}$, where $m, x, \ell, y$ are given in \textbf{Claim \ref{c.r=1}} are disjoint with each other. Now, we shall show that the number of elements in the union of sets given in \textbf{Claim \ref{c.r=1}} is $ n^2 - 1$. If $m = a, a-1, \ldots, 1$ then $x \in \{0, 1, 2\}, \; x \in \{0,1, 2, 3, 4, 5\}, \ldots, x \in \{0, 1, 2, \ldots, n-2 \}$, respectively. Thus, the total number of the sets of the form $\mathcal A_{m,x}$ is $3 + 6 + \cdots + n-1 = \frac{(n+2)(n-1)}{6}$. Similarly,  the total number of  sets of the form $\mathcal B_{\ell,y}$ is $2 + 5 + \cdots + n-2 = \frac{n(n - 1)}{6}$. Consequently, we  get $|(\cup \mathcal A_{m,x})\cup(\cup \mathcal B_{\ell,y})|= n^2 - 1$. Note that the set  $C = \{(a+1, a+1)\}$ is disjoint with $\mathcal A_{m,x}$ and $\mathcal B_{\ell,y}$.  Thus, the sets $\mathcal A_{m,x}, \; \mathcal B_{\ell,y}$ and $C$ forms a partition of $V(\c)$. Therefore, $\chi (\Delta (B_n))\leq{\frac{n^2 - 1}{3}} +1 = \left \lceil \frac{n^2}{3}\right\rceil$. Hence, $\chi (\Delta (B_n)) = \left \lceil \frac{n^2}{3}\right\rceil$.\\
\textbf{Case 3:}  $n = 3a + 2$. By the similar concept used in \textbf{Case 1}, the sets $\mathcal A_{m,x}$ and $\mathcal B_{\ell,y}$, where $m, x, \ell, y$ are given in \textbf{Claim \ref{c.r=2}} are disjoint with each other.  Now, we shall show that the number of elements in the union of above defined sets is $ n^2 - 1$. If $m = a +1, a, \ldots, 1$ then $x \in \{0\}, \; x \in \{0,1, 2, 3\}, \ldots, x \in \{0, 1, 2, \ldots, n-2 \}$, respectively. Thus, the total number of  sets of the form $\mathcal A_{m,x}$ is $1 + 4 + \cdots + n-1 = \frac{n(n + 1)}{6}$. Similarly,  the total number of the sets of the form $\mathcal B_{\ell,y}$ is $3 + 6 + \cdots + n-2 = \frac{(n + 1)(n - 2)}{6}$. Consequently, we  get $|(\cup \mathcal A_{m,x})\cup(\cup \mathcal B_{\ell,y})|= n^2 - 1$. Note that the set  $C = \{(2a+2, 2a+2)\}$ is disjoint with $\mathcal A_{m,x}$ and $\mathcal B_{\ell,y}$.  Thus, the sets $\mathcal A_{m,x}, \; \mathcal B_{\ell,y}$ and $C$ forms a partition of $V(\c)$. Therefore, $\chi (\Delta (B_n))\leq{\frac{n^2 - 1}{3}} + 1 = \left \lceil \frac{n^2}{3}\right\rceil$ so that  $\chi(\c) = \left \lceil \frac{n^2}{3}\right\rceil$.
\end{proof}

\begin{theorem}\label{vertex-connectivity}
For $n \ge 3$, the vertex connectivity of $\c$ is $n(n-2)$.
\end{theorem}

\begin{proof} By Theorem 4.1.9 of \cite{b.West} and Corollary \ref{degree of all B_n}, we have $\kappa (\c) \leq n(n-2)$.  By Menger's theorem  (cf. \cite[Theorem 3.2]{b.bondy1976graph}), to prove another inequality, it is sufficient to show that there exist at least $n(n - 2)$ internally disjoint paths between arbitrary pair of vertices. Let $(a, b)$ and $(c, d)$ be arbitrary pair of vertices in $V(\c)$. Now consider
	\[A = \{(b,x) : \; x \in [n] \} \cup \{(x, a) : \; x \in [n] \} \]
	and
	\[B = \{(d,x) : \; x \in [n] \} \cup \{(x, c) : \; x \in [n] \}. \]
Note that $|A| = |B| = 2n-1$ and each element of $A$ and $B$ is not adjacent with $(a, b)$ and $(c, d)$, respectively (see Remark \ref{r.non adjacent vertices}). If $T = A \cup B \cup \{(a, b), (c, d) \}$, then note that every element of $T' = V(\c) \setminus T$, commutes with $(a, b)$ and $(c, d)$. Thus, for each element $(x, y)$ of  $T'$, we have a path $(a, b) \sim (x, y) \sim (c, d)$. Consequently, there are at least $|T'|$ many internally disjoint paths between $(a, b)$ and $(c, d)$. We show that there exist $n(n-2)$ internally disjoint paths between $(a, b)$ and $(c, d)$ in the following cases.\\
\textbf{Case 1:} Both $(a, b)$ and $(c, d)$ are distinct idempotents. Clearly $a = b, c= d$ and $a \ne c$. Then, we have $A \cap B = \{(a, c), (c, a) \}$ so that $|T'| = n^2 - 4n + 4$. As a consequence, we get  $n^2 - 4n + 4$ internally disjoint paths between $(a, b)$ and $(c, d)$. Furthermore, for $x \in [n] \setminus \{a, c \}$,  we have
$ (a, a) \sim (c, x) \sim (a, x) \sim (c, c) \; \text{and} \;  (a, a) \sim (x, c) \sim (x, a) \sim (c, c)$
internally disjoint paths between $(a, b)$ and $(c, d)$  which are $2n - 4$  in total.  Thus, there are at least $n^2 - 2n$ internally disjoint paths between $(a, b)$ and $(c, d)$.\\
\textbf{Case 2:} Either $(a, b)$ or $(c,d)$ is idempotents. Without loss of generality, let $c = d$. Further, we have the following subcases.

\textit{Subcase 2.1:} $c \ne a, b$. Then $A \cap B = \{(b, c), (c, a) \}$ so that $|T'| = n^2 - 4n + 3$. Consequently, we get $n^2 - 4n + 3$ internally disjoint paths between $(a, b)$ and $(c, d)$. In addition to that, for $x \in [n] \setminus \{a, b, c \}$,  we have
\[ (a, b) \sim (c, x) \sim (b, x) \sim (c, c), \]
\[ (a, b) \sim (x, c) \sim (x, a) \sim (c, c)  \]
internally disjoint paths between $(a, b)$ and $(c, d)$  which are $2n - 6$  in total. Further, we have three more  paths between $(a, b)$ and $(c, d)$ as follows:
\[(a, b) \sim (c, b) \sim (a, a) \sim (c, c),\]
\[(a, b) \sim (a, c) \sim (b, b) \sim (c, c),\]
\[(a, b) \sim (c, c).\]
Thus, there are at least $n^2 - 2n$ internally disjoint paths between $(a, b)$ and $(c, d)$.

\textit{Subcase 2.2:} $c= a$ or $c = b$. First suppose $c = a$. Then, we have $A \cap B = \{(x, a) : x \in [n] \}$ so that $|T'| = n^2 - 3n + 2$. Therefore, $\c$ contains  $n^2 - 3n + 2$ internally disjoint paths between $(a, b)$ and $(c, d)$. Additionally, for $x \in [n] \setminus \{a, b \}$,  we have $n - 1$ internally disjoint paths
$ (a, b) \sim (a, x) \sim (b, x) \sim (a, a)$
between $(a, b)$ and $(c, d)$. Thus, there are at least $n^2 - 2n$ internally disjoint paths between $(a, b)$ and $(c, d)$. Similarly, for $ c =b$, at least $n^2 - 2n$ internally disjoint paths between $(a, b)$ and $(c, d)$ can be obtained.\\
\textbf{Case 3:} Both $(a, b)$ and $(c, d)$ are non-idempotent element. Clearly, $a \ne b$ and $c \ne d$. Further, we have the following subcases.

\textit{Subcase 3.1:} $a, b, c, d$  all are distinct. Then, we have $A \cap B = \{(b, c), (d, a) \}$ so that $|T'| = n^2 - 4n + 2$. Thus, there are $n^2 - 4n + 2$ internally disjoint paths between $(a, b)$ and $(c, d)$. In addition to that, for $x \in [n] \setminus \{a, b, c, d \}$,  we have $(a, b) \sim (x, c) \sim (x, a) \sim (c, d) \; \text{and}  \; (a, b) \sim (d, x) \sim (b, x) \sim (c, d)$
internally disjoint paths between $(a, b)$ and $(c, d)$  which are $2n -8$  in total. Moreover, we have six additional  paths between $(a, b)$ and $(c, d)$ as follows:
\[(a, b) \sim (a, c) \sim (b, b) \sim (c, d),\]
\[(a, b) \sim (c, c) \sim (b, d) \sim (c, d),\]
\[(a, b) \sim (d, c) \sim (a,a) \sim (c, d),\]
\[(a, b) \sim (d, d) \sim (b, a) \sim (c, d),\]
\[(a, b) \sim (d, b) \sim (c, a) \sim (c, d),\]
\[(a, b) \sim (c, d).\]
Thus, there are at least $n^2 - 2n$ internally disjoint paths between $(a, b)$ and $(c, d)$.\\

\textit{Subcase 3.2:} $c \in \{a, b\}$. If $c = a$, then $A \cap B = \{(x, a) \; : \; x \in [n] \}$ so that $|T'| = n^2 - 3n$. Therefore, $\c$ contains  $n^2 - 3n$ internally disjoint paths between $(a, b)$ and $(c, d)$. Additionally, for $x \in [n] \setminus \{a, b, d \}$,  we have
$ (a, b) \sim (d, x) \sim (b, x) \sim (a, d)$
internally disjoint paths between $(a, b)$ and $(c, d)$  which are $n - 3$  in total. Besides these paths, we have three  paths between $(a, b)$ and $(c, d)$ as follows:
\[(a, b) \sim (d, b) \sim (a, a) \sim (b, d) \sim (a, d),\]
\[(a, b) \sim (d, d) \sim (b, b) \sim (a, d),\]
\[(a, b) \sim (a, d).\]
Thus, there are at least $n^2 - 2n$ internally disjoint paths between $(a, b)$ and $(c, d)$. On the other hand $c= b$. Now we have the two possibilities (i) $d = a$ (ii) $a, b, d$ are distinct. If $d = a$, then $A \cap B = \{(b, b), (a, a) \}$ so that $|T'| = n^2 - 4n + 4$. Consequently, we get $n^2 - 4n + 4$ internally disjoint paths between $(a, b)$ and $(c, d)$. In addition to that, for $x \in [n] \setminus \{a, b \}$,  we have
$ (a, b) \sim (x, b) \sim (x, a) \sim (b, a) \; \text{and}  \; (a, b) \sim (a, x) \sim (b, x) \sim (b, a)$
internally disjoint paths between $(a, b)$ and $(c, d)$  which are $2n -4$  in total. Thus, we get at least  $n^2 -2n$ internally disjoint paths between $(a, b)$ and $(c, d)$. For distinct $a, b$ and $d$, we get $A \cap B = \{(d, a), (b, b) \}$ so that $|T'| = n^2 - 4n + 4$. Consequently, we get  $n^2 - 4n + 4$ internally disjoint paths between $(a, b)$ and $(c, d)$. Additionally, for $x \in [n] \setminus \{a, b,  d \}$,  we have $2n - 6$ internally disjoint paths
\[(a, b) \sim (x, b) \sim (x, a) \sim (b, d),\]  \[(a, b) \sim (d, x) \sim (b, x) \sim (b, d)\]
 between $(a, b)$ and $(c, d)$. Besides these paths, we have two more paths
$(a, b) \sim (d, b) \sim (a, a) \sim (b, d)$ and $(a, b) \sim (d, d) \sim (b, a) \sim (b, d)$. Thus, there are at least $n^2 - 2n$ internally disjoint paths between $(a, b)$ and $(c, d)$.

\textit{Subcase 3.3:} $d \in \{a, b\}$. If $d = a$, then $A \cap B = \{(b, c), (a, a) \}$ so that $|T'| = n^2 - 4n + 4$. Consequently, we get $n^2 - 4n + 4$ internally disjoint paths between $(a, b)$ and $(c, d)$. In addition to that, for $x \in [n] \setminus \{a, b, c \}$,  we have
$ (a, b) \sim (a, x) \sim (b, x) \sim (c, a)$ and $ (a, b) \sim (x, c) \sim (x, a) \sim (c, a)$ internally disjoint paths between $(a, b)$ and $(c, d)$  which are $2n - 6$  in total. Moreover, we have two paths
$(a, b) \sim (a, c) \sim (b, b) \sim (c, a) \; \text{and} \; (a, b) \sim (c, c) \sim (b, a) \sim (c, a)$
between $(a, b)$ and $(c, d)$. Thus, there are at least $n^2 - 2n$ internally disjoint paths between $(a, b)$ and $(c, d)$. On the other hand, let $d = b$. Then $A \cap B = \{(b, x) \; : \; x \in [n] \}$ so that $|T'| = n^2 - 3n$. As a consequence, we get  $n^2 - 3n$ internally disjoint paths between $(a, b)$ and $(c, d)$. Furthermore, for $x \in [n] \setminus \{a,b,c \}$,  we have $n - 3$ internally disjoint paths
$(a, b) \sim (x, c) \sim (x, a) \sim (c, b)$
between $(a, b)$ and $(c, d)$. Besides these paths, we have three more  paths between $(a, b)$ and $(c, d)$ as follows:
\[(a, b) \sim (c, c) \sim (a, a) \sim (c, b),\]
\[(a, b) \sim (a, c) \sim (b, b) \sim (c, a) \sim (c, b),\]
\[(a, b) \sim (c, b).\]
Thus, there are at least $n^2 - 2n$ internally disjoint paths between $(a, b)$ and $(c, d)$.
\end{proof}

In view of Lemma \ref{nbd} and since $\kappa(\mathcal{G}) \le \kappa'(\mathcal{G}) \le \delta(\mathcal{G})$, we have the following corollary.
\begin{corollary} For $n \geq 3$, the edge connectivity of $\c$ is $n(n - 2)$.
\end{corollary}

\section{Algebraic properties of $\c$}

In order to study algebraic aspects of the graph $\c$, in this section we obtain automorphism group (see Theorem \ref{automorphism group}) and endomorphism monoid (see Theorem \ref{End}) of $\c$.

\subsection{Automorphism group of $\c$}
An \emph{automorphism} of a graph $\mathcal{G}$ is a permutation $f$ on  $V(\mathcal{G})$ with the property that, for any vertices $u$ and $v$, we have $uf \sim vf$ if and only if $u \sim v$. The set ${\rm Aut}(\mathcal{G})$ of all graph automorphisms of a graph $\mathcal{G}$ forms a group with respect to  composition of mappings. The symmetric group of degree $n$ is denoted by $S_n$. For $n = 1$, the group Aut$(\c)$ is trivial. For the remaining subsection, we assume $n \geq 2$.   

\begin{lemma}\label{idem mapsto idem}
Let $x \in V(\c)$ and $f \in$ {\rm Aut}$(\c)$. Then $x$ is an idempotent if and only if $xf$ is an idempotent.	
\end{lemma}

\begin{proof}
Since $f$ is an automorphism, we have deg$(x) =$deg$(xf)$. By Corollary \ref{degree of all B_n}, the result holds.
\end{proof}

\begin{lemma}\label{nonidempotent mapping}
For $f \in {\rm Aut}(\c)$ and $i, j, k, k' \in [n]$ such that $(i, i)f = (k, k)$ and $(j, j)f = (k', k')$, we have either $(i, j)f = (k, k')$ or $(i, j)f = (k', k)$.
\end{lemma}

\begin{proof}
	For $i \ne j$, suppose that $(i, j)f = (x, y)$. Clearly, $(i, j) \nsim (i, i)$ so that
	$(x, y) = (i, j)f \nsim (i, i)f = (k, k)$.
	Since $(x, y) \nsim (k, k)$, we get either $x = k$ or $y = k$. Similarly, for $(i, j) \nsim (j, j)$, we have either $x = k'$ or $y = k'$.
	Thus, by Lemma \ref{idem mapsto idem}, we have $(x, y) = (k, k')$ or $(x, y) = (k', k)$.
\end{proof}

\begin{lemma}\label{phi-sigma}
	For $\sigma \in S_n$, let $\phi_{\sigma}:V(\c) \rightarrow V(\c)$ defined by $(i, j) \phi_{\sigma} = (i \sigma, j \sigma)$. Then $\phi_{\sigma} \in {\rm Aut}(\c)$.
\end{lemma}

\begin{proof}
	It is easy to verify that $\phi_{\sigma}$ is a permutation on $V(\c)$. Now we show that $\phi_\sigma$ preserves adjacency. Let $(i, j), (x, y) \in V(\c)$  such that
	$(i, j) \sim (x, y)$. Now,
	\begin{align*}
	(i, j) \sim (x, y) & \Longleftrightarrow x \ne j \; \text{and} \; y \ne i \\
	& \Longleftrightarrow \; \text{for} \; \sigma \in S_n, \; \text{we have} \; x \sigma \ne j \sigma \; \text{and} \; y \sigma \ne i \sigma \\
	& \Longleftrightarrow (i \sigma, j \sigma) \sim (x \sigma, y\sigma) \\
	& \Longleftrightarrow (i, j)\phi_{\sigma} \sim (x, y) \phi_{\sigma}.
	\end{align*}
	Hence, $\phi_{\sigma} \in {\rm Aut}(\c)$.
\end{proof}

\begin{lemma}\label{alpha}
	Let $\alpha : V(\c) \rightarrow V(\c)$ be a mapping defined by $(i, j)\alpha = (j, i)$. Then $\alpha \in {\rm Aut}(\c)$.
\end{lemma}

\begin{proof}
	It is straightforward to verify that $\alpha$ is a one-one and onto map on $V(\c)$. Note that
	\begin{align*}
	(i, j) \sim (x, y) & \Longleftrightarrow x \ne j \; \text{and} \; y \ne i \\
	& \Longleftrightarrow (j, i) \sim (y, x) \\
	& \Longleftrightarrow (i, j)\alpha \sim (x, y)\alpha.
	\end{align*}
	Hence, $\alpha \in {\rm Aut}(\c)$.
\end{proof}

\begin{remark}
For  $\phi_{\sigma}$ and $\alpha$, defined in Lemma \ref{phi-sigma} and \ref{alpha}, we have $\phi_{\sigma} \circ \alpha = \alpha \circ \phi_{\sigma}$.
\end{remark}

\begin{proposition}\label{aut-sigma or comp}
For each $f \in {\rm Aut}(\c)$, we have either $f = \phi_{\sigma}$ or $f = \phi_{\sigma} \circ \alpha$ for some $\sigma \in S_n$.
\end{proposition}
  \begin{proof}
Since $f \in {\rm Aut}(\c)$,  by Lemma \ref{idem mapsto idem}, note that there exists  a permutation $\sigma : [n] \rightarrow [n]$ such that $i\sigma = j$ $\iff (i,i)f = (j, j)$, determined by $f$. Thus, we have $(i, i)f = (i\sigma, i\sigma)$ for all $i \in [n]$. Let $j \ne i$. Then by Lemma \ref{nonidempotent mapping}, we get either $(i, j)f = (i\sigma, j\sigma)$ or $(i, j)f = (j\sigma, i\sigma)$.
\\
\textbf{Case 1:} $(i, j)f = (i\sigma, j\sigma)$. We show that for any $(k, l) \ne (i, j)$, where $k \ne l$, we have $(k, l)f = (k\sigma, l\sigma)$ so that $f = \phi_{\sigma}$. We have the following subcases:

 \textit{Subcase 1.1:} $k = i$. Clearly, $l \ne j$. Then $(i, j) \sim (k, l)$ so that $(i\sigma, j\sigma) = (i, j)f \sim (k, l)f$. We must have $(k, l)f = (k\sigma, l\sigma)$.
 
 \textit{Subcase 1.2:} $l = j$. Clearly, $k \ne i$. Then $(i, j) \sim (k, l)$ so that $(i\sigma, j\sigma) = (i, j)f \sim (k, l)f$. We must have $(k, l)f = (k\sigma, l\sigma)$.
 
 \textit{Subcase 1.3:} $l = i$. Note that $(i, j) \nsim  (k, l)$ so that $(i\sigma, j\sigma) = (i, j)f \nsim  (k, l)f$. We must have $(k, l)f = (k\sigma, l\sigma)$.
 
\textit{Subcase 1.4:} $k = j$. Note that $(i, j) \nsim  (k, l)$ so that $(i\sigma, j\sigma) = (i, j)f \nsim  (k, l)f$. We must have $(k, l)f = (k\sigma, l\sigma)$.

\textit{Subcase 1.5:} $k, l \in [n]\setminus \{i, j\}$. By \textit{Subcase 1.1}, we get $(i, l)f= (i\sigma, l\sigma)$. Thus, by \textit{Subcase 1.2} we get $(k, l)f = (k\sigma, l\sigma)$.
\\
\textbf{Case 2:} $(i, j)f = (j\sigma, i\sigma)$. Let, if possible, there exists $(k, l) \ne (i, j)$, where $k \ne l$, such that $(k, l)f = (k\sigma, l\sigma)$. Then by \textbf{Case 1}, we get $(i, j)f = (i\sigma, j\sigma)$. Consequently, $i =  j$; a contradiction. Thus, for any $(k, l) \ne (i, j)$, we have $(k, l)f = (l\sigma, k\sigma)$ so that $f = \phi_{\sigma} \circ \alpha$.
\end{proof}

\begin{theorem}\label{automorphism group}
	For $n \ge 2$, we have ${\rm Aut}(\c) \cong S_n \times \mathbb{Z}_2$. Moreover, $|{\rm Aut}(\c)| = 2(n!)$.
\end{theorem}

\begin{proof}
	In view of Lemmas \ref{phi-sigma}, \ref{alpha} and \ref{aut-sigma or comp}, note that the underlying set of the automorphism group of $\c$ is
	\[{\rm Aut}(\c) = \{\phi_{\sigma} \; : \; \sigma \in S_n \} \cup \{\phi_{\sigma} \circ \alpha \; : \; \sigma \in S_n \},\] where $S_n$ is a symmetric group of degree $n$. Note that the groups Aut$(\c)$ and $S_n \times \mathbb Z_2$ are isomorphic under the assignment $\phi_{\sigma} \mapsto (\sigma, \bar{0})$ and $\phi_{\sigma} \circ \alpha \mapsto (\sigma, \bar{1})$.  Since, all the elements in ${\rm Aut}(\c)$ are distinct, we have $|{\rm Aut}(\c)| = 2(n!)$.
\end{proof}

\subsection{Endomorphism monoid of $\c$}
A mapping $f$ from a graph $\mathcal G$ to $\mathcal G'$ is said to be a \emph{homomorphism} if $x \sim y$, then $xf \sim yf$  for all $x, y \in V(\mathcal G)$. If $\mathcal G' = \mathcal G$, then we say $f$ is an \emph{endomorphism}. Note that the set {\rm End}$(\mathcal G)$ of all endomorphisms on $\mathcal G$ forms a monoid with respect to the composition of mappings. First we obtain the endomorphism monoid of $\c$ for $n \in \{2, 3\}$. The following remark is useful in the sequel.

\begin{remark}\label{r.image-clique-maximum-size}
Let $f \in {\rm End}(\mathcal G)$ and  $K$ be a clique of maximum size in $\mathcal G$. Then $Kf$ is again a clique of maximum size.
\end{remark}

\begin{lemma}
\rm{End}$(\Delta(B_2)) = \{f: V(\Delta(B_2)) \rightarrow  V(\Delta(B_2)) \; : \; \mathcal Ef = \mathcal E \}$, where $\mathcal E = \{(1, 1), (2, 2)\}$.
\end{lemma}

\begin{proof}
For $x, y \in V(\Delta(B_2))$, note that $x \sim y $ if and only if $x, y$ belongs to $\mathcal E$. Hence, we have the result.
\end{proof}

For $\sigma \in S_3$, we define the mappings $f^{\sigma}$ and $g^{\sigma}$ on $V(\Delta(B_3))$ by\\
$\bullet \; (i, i) \overset{ f^{\sigma}}{\longmapsto}(i\sigma, i\sigma), (1, 2)\overset{ f^{\sigma}}{\longmapsto} (1\sigma, 1\sigma), (1, 3)\overset{ f^{\sigma}}{\longmapsto} (3\sigma, 3\sigma), (2, 3)\overset{ f^{\sigma}}{\longmapsto} (2\sigma, 2\sigma), (2, 1)\overset{ f^{\sigma}}{\longmapsto} (1\sigma, 1\sigma), (3, 1)\overset{ f^{\sigma}}{\longmapsto} (3\sigma, 3\sigma)$, $(3, 2)\overset{ f^{\sigma}}{\longmapsto} (2\sigma, 2\sigma)$, and\\
$\bullet \; (i, i)\overset{ g^{\sigma}}{\longmapsto} (i\sigma, i\sigma), (1, 2)\overset{ g^{\sigma}}{\longmapsto}(2\sigma, 2\sigma), (3, 2)\overset{ g^{\sigma}}{\longmapsto} (3\sigma, 3\sigma), (3, 1)\overset{ g^{\sigma}}{\longmapsto}(1\sigma, 1\sigma), (2, 1)\overset{ g^{\sigma}}{\longmapsto} (2\sigma, 2\sigma), (2, 3) \overset{ g^{\sigma}}{\longmapsto}(3\sigma, 3\sigma)$, $(1, 3)\overset{ g^{\sigma}}{\longmapsto} (1\sigma, 1\sigma)$, respectively.\\
It is routine to verify that $f^{\sigma}, g^{\sigma} \in$ End$(\Delta(B_3))$.

\begin{lemma}
\rm{End}$(\Delta(B_3)) =$ \rm{Aut}$(\Delta(B_3)) \cup \{f^{\sigma} \; : \; \sigma \in S_3 \} \cup \{g^{\sigma} \; : \; \sigma \in S_3 \}$, where $f^{\sigma}$ and $g^{\sigma}$ are the endomorphisms on $V(\Delta(B_3))$ as defined above.
\end{lemma}

\begin{proof}
Let $\psi \in$ End$(\Delta(B_3))$. By Figure \ref{B_3}, note that  $\{(1,1), (2, 2), (3, 3)\}$ is the only clique of maximum size in $\Delta(B_3)$. Since the image of a clique of maximum size under an endomorphism is again a clique of maximum size, we get $(i, i)\psi$ is an idempotent element for all $i \in \{1, 2, 3\}$. Also note that restriction of $\psi$ to $\mathcal E = \{(1, 1), (2, 2), (3, 3)\}$ is a bijective map from $\mathcal E$ to $\mathcal E$. If $(i, i)\psi = (j, j)$ for some $j \in \{1, 2, 3\}$, then define $\sigma : \{1, 2, 3\} \rightarrow \{1, 2, 3\}$ by $i\sigma = j$. Consequently, $\sigma \in S_3$. Suppose $(i,j)\psi$ is an idempotent element for some distinct $i, j \in \{1, 2, 3\}$. Without loss of generality, let $i = 1$ and $j = 2$. Since $(1, 2) \sim (3, 3)$ we have $(1,2)\psi \sim (3, 3)\psi = (3\sigma, 3\sigma)$. Consequently, $(1, 2)\psi \in \{(1\sigma, 1\sigma), (2\sigma, 2\sigma)\}$. If $(1, 2)\psi = (1\sigma, 1\sigma)$, then $\psi = f^{\sigma}$. Otherwise, $\psi = g^{\sigma}$. On the other hand, if $(i, j)\psi$ is a non-idempotent for all $i \neq j$. Let $(i, j)\psi = (x, y)$, where $x \ne y$. For $k \ne i, j$, we have $(x, y) = (i, j)\psi \sim(k, k)\psi$. Thus, $(i, j)\psi$ is either $(i\sigma, j\sigma)$ or  $(j\sigma, i\sigma)$. By the similar argument used in Proposition \ref{aut-sigma or comp}, we have $\psi \in $ Aut$(\Delta(B_3))$.
\end{proof}

Now, we obtain End$(\Delta(B_n))$ for $n \geq 4$. We begin with few definitions and necessary results. If $\mathcal G'$ is a subgraph of $\mathcal G$, then a homomorphism $f : \mathcal G \rightarrow \mathcal G'$ such that $xf = x$ for all $x \in \mathcal G'$ is called a \emph{retraction} of $\mathcal G$ onto $\mathcal G'$ and $\mathcal G'$ is said to be a \emph{retract} of $\mathcal G$. A subgraph $\mathcal G'$ of $\mathcal G$ is said to be a \emph{core} of $\mathcal G$ if and only if it admits no proper retracts (cf. \cite{a.1992Hell}). Let $X \subset A, \; Y \subseteq B$ and  $f$ be any mapping from the set $A$ to $B$ such that $Xf \subseteq Y$. We write the \emph{restriction map} of $f$ from $X$ to $Y$ as $f_{X \times Y}$ i.e $f_{X \times Y} : X \rightarrow Y$ such that $xf_{X \times Y} = xf$.

\begin{proposition}[{ \cite[Proposition 2.4]{a.core}}]\label{core}
A graph $\mathcal G$ is a core if and only if \rm{End}$(\mathcal G) =$ \rm{Aut}$(\mathcal G)$. 
\end{proposition}

\begin{lemma}\label{maps-nonidempotent-B_4}
Let $f$  be a retraction of $\Delta(B_4)$. Then a non-idempotent element  maps to a non-idempotent element of $B_4$ under $f$.
\end{lemma}

\begin{proof}
Let, if possible there exists a non-idempotent element $(i, j)$ of $B_4$ such that $(i, j)f$ is an idempotent element. In order to get a contradiction, first we show that $(a, b)f   \in \mathcal E = \{(1, 1), (2, 2), (3, 3), (4, 4)\}$ for all $a \ne b \in \{1, 2, 3, 4\}$. Without loss of generality, we may assume that $i = 1$ and $j = 2$.  In view of Remark \ref{r.clique} , any clique $K $ in $\Delta(B_4)$ of maximum size is either $K = \mathcal E$ or $K = A \times B$, where $A$ and $B$ are disjoint subsets of $\{1, 2, 3, 4\}$ of size two. Therefore, $\Delta(B_4)$ has two cliques of maximum size which contains $(1, 2)$ viz. $K_1 = \{1,3\} \times \{2, 4\}$ and $K_2 = \{1,4\} \times \{2, 3\}$. Note that for disjoint subsets $A$ and $B$ of $\{1, 2, 3, 4\}$, the clique $A \times B$ does not contain an idempotent element. Since $(1, 2)f$ is an idempotent element and by Remark \ref{r.image-clique-maximum-size}, we have $K_1f = K_2f = \mathcal E$. By using the other elements of $\left(K_1f \cup K_2f\right) \setminus \{(1, 2)f\}$,  in a similar manner, one can observe that the image of remaining non-idempotent elements belongs to $\mathcal E$. Thus, $(a, b)f   \in \mathcal E$ for all $a \ne b \in [n]$. Now, we show that for any two distinct $x, y \in \{1, 2, 3, 4\}$, $(x, y)f$ is either $(x, x)$ or $(y, y)$. Since image of non-idempotent element is an idempotent so that $(x, y)f = (p, p)$ for some $p \in \{1, 2, 3, 4\}$. Note that $p \in \{x, y\}$. Otherwise, $(p, p) \sim (x, y)$ implies $(p, p)=(p, p)f \sim (x, y)f = (p, p)$; which is not possible. Now suppose $(1, 2)f = (1, 1)$. Since $(1, 2) \sim (1, k)$ for $k \ne 1, 2$, we get  $(1, 1) = (1, 2)f \sim (1, k)f$. Consequently, $(1, k)f = (k, k)$. Similarly, we get $(2, k)f = (2, 2)$. Therefore,  $(2, 3)f = (2, 4)f = (2, 2)$. We get a contradiction as $(2, 4) \sim (2, 3)$. Similarly, we get a contradiction when $(1, 2)f = (2, 2)$. Hence, the result hold.
\end{proof}

\begin{lemma}\label{maps-nonidempotent}
For $n \geq 5$, let $f \in$  {\rm End}$(\c)$. Then a non-idempotent element  maps to a non-idempotent element of $B_n$ under $f$.
\end{lemma}

\begin{proof}
Let $(i, j)$ be a non-idempotent element of $B_n$. By Remark \ref{non-idempotent_clique-II}, there exists a clique $K$ of maximum size which contains $(i,j)$. In view of Remarks \ref{non-idempotent_clique} and \ref{r.image-clique-maximum-size}, all the elements of $Kf$ are non-idempotent. Thus, $(i, j)f$ is a non-idempotent element.
\end{proof}

\begin{proposition}\label{retract-idempotents}
For $n \geq 4$, let $\mathcal G'$ be a retract of $\c$ such that $(i, i) \in \mathcal G'$ for all $i \in [n]$. Then $\mathcal G' = \c$.
\end{proposition}

\begin{proof}
Since $\mathcal G'$ is a retract of $\c$, there exists a homomorphism $f : \c \rightarrow \mathcal G'$ such that $xf = x$ for all $x \in V(\mathcal G')$. Let $(i, j)$ be a non-idempotent element of $B_n$. Then $(i, j)f$ is a non-idempotent element of $B_n$ (cf. Lemmas \ref{maps-nonidempotent-B_4} and \ref{maps-nonidempotent}). Let $(i, j)f = (x, y)$, where $x \neq y$. For $k \in [n] \setminus \{i, j\}$, we have $(i, j) \sim (k, k)$. Since $(k, k) \in \mathcal G'$, we get $(x, y) \in N [(k,k)]$. By Lemma \ref{nbd}(i), $x, y \neq k$.	
Consequently, $(x, y) \in \{ (i, j), (j, i)\}$. Thus, either $(i, j)f = (i, j)$ or $(j, i)$. Now to prove $\mathcal G' = \c$, we show that $f$ is an identity map. Since $(i, i) \in \mathcal G'$, it is sufficient to prove that for any $i, j \in [n]$ such that $i \neq j$, we have $(i,j)f = (i, j)$. Let if possible, $(i, j)f = (j, i)$ for some $i \neq j$. Then $(j, i)f = (j, i)$. For $p \in [n] \setminus \{i, j\}$, note that $(j, p)f = (j, p)$ because if $(j, p)f = (p, j)$, then $(j, p) \sim (j, i)$ implies $(j, p)f = (p, j) \nsim (j, i) = (j, i)f$; a contradiction. Further, note that $(i, p)f \notin \{(i, p), (p, i)\}$ which is not possible. For instance, if $(i, p)f = (i, p)$ then $(i, p) \sim (i, j)$ gives $(i, p)f \sim (i, j)f$. Consequently, we get $(i,p) \sim (j, i)$; a contradiction. On the other hand, if $(i, p)f = (p, i)f$ then $(i, p) \sim (j, p)$ gives $(i, p)f = (p, i) \nsim (j, p) = (j, p)f$; a contradiction. Hence, $f$ is an identity map so that $\mathcal G' = \c$.      
\end{proof}

To obtain the {\rm End$(\c)$}, following Lemmas will be useful.

\begin{lemma}\label{retraction-maximum-size-clique}
For $n \geq 4$, let $f$ be a retraction of $\c$ onto $\mathcal G'$. Then there exists a clique $K$ of maximum size in $\mathcal G'$  such that $K =  A \times B$ where $A$ and $B$ forms a partition of  $[n]$. Moreover,
\begin{enumerate}[\rm (i)]
\item if $n$ is even then $|A| = |B| = \frac{n}{2}$, or
\item if $n$ is odd then either $|A| =\frac{n - 1}{2}, |B| = \frac{n + 1}{2}$ or  $|A| = \frac{n + 1}{2}, |B|= \frac{n - 1}{2}$.
\end{enumerate}
\end{lemma}

\begin{proof}
Let $f$ be a retraction on $\c$. For $n \geq 4$, in view of Corollary \ref{clique-non-idempotent}, Lemma \ref{clique-4} and Theorem \ref{clique equality}, $\c$ contains  a clique $K'$ of maximum size such that all the elements of $K'$ are non-idempotent.  By  Remark \ref{r.image-clique-maximum-size} and Lemmas  \ref{maps-nonidempotent-B_4}, \ref{maps-nonidempotent}, $K'f$ is a clique of maximum size and all of its elements are non-idempotents. Now consider $K'f= K$, by the proof of Lemma \ref{clique inequality}, we get $K = A \times B$ where $A$ and  $B$ forms a partition of $[n]$ together with (i) or (ii).
\end{proof}

%

In the following lemma, we provide the possible images of non-idempotent elements of $B_n$ under a retraction.

\begin{lemma}\label{retraction-image-nonidempotent}
Let $f$ be a retraction of $\c$ onto $\mathcal G'$, where $n\geq 4$. Then for $p \ne q \in [n]$, we have \[(p, q)f \in \{(t, p) : t \in A \} \cup \{(q, t) : t \in B \} \cup \{(p, q)\},\] for some partition $\{A, B\}$ of $[n]$. Moreover, 
\begin{enumerate}[\rm (i)]
\item if $p \in A$, then $(p, q)f \ne (t, p)$ for any $t \in A$.
\item if $q \in B$, then $(p, q)f \ne (q, t)$ for any $t \in B$.
\end{enumerate}
\end{lemma}

\begin{proof}
In view of Lemma \ref{retraction-maximum-size-clique}, there exists a clique $K = A \times B$ of maximum size in $\mathcal G'$ for some partition $\{A, B\}$ of $[n]$. Suppose $(p, q)f = (x, y)$. Then, by Lemmas \ref{maps-nonidempotent-B_4} and \ref{maps-nonidempotent}, we have $x \ne y$. If $(p, q)f = (p, q)$ then there is nothing to prove. Now let $(p, q)f = (x, y)$ where $(x, y) \ne (p, q)$. If $x, y \notin \{p, q\}$, then $(p, q) \sim (x, y)$ gives $(p, q)f = (x, y)f = (x, y)$; a contradiction. Then either $x \in \{p, q\}$ or $y \in \{p, q\}$. If $x = p$, then clearly $y \notin \{p, q\}$. Consequently, $(p, q)  \sim (x, y)$ provides again a contradiction. Therefore, $x \ne p$. Similarly, one can  show that $y \ne q$. It follows that $(p, q)f = (x, y)$ where either $x = q$ or $y = p$. Now observe that if $y = p$, then $x \in A$. If possible, let $x \in B$. Then  for $\alpha \in A \setminus \{q\}$,  $(\alpha, x)f = (\alpha, x)$ as $(\alpha, x) \in A \times B \subseteq \mathcal G'$. Since $x \ne p$ as $ x \ne y$, we get $(p, q) \sim (\alpha, x)$ so that $(p, q)f = (x, p) \sim (\alpha, x)  = (\alpha, x)f$; a contradiction of Remark \ref{r.non adjacent vertices}. In a similar manner it is not difficult to observe if $x = q$, then $y \in B$.

 To prove addition part of the lemma, suppose $p \in A$ and $(p, q)f = (t, p)$ for some $t \in A$. For $r \in B$ such that $r \ne q$, we have $(p, q) \sim (p, r)$ and $(p, r)f = (p, r)$ as $(p, r) \in K \subseteq \mathcal G'$. Consequently, we get $(p, q)f = (t, p) \sim (p, r) = (p, r)f$; a contradiction of Remark \ref{r.non adjacent vertices}. Thus, $(p, q)f \ne (t, p)$. Using similar argument, observe that for $q \in B$, $(p, q)f \ne (q, t)$ for any $t \in B$. Thus, the result hold. 
\end{proof}

\begin{theorem}
For $n = 4$, we have \rm{End}$(\c)$ = \rm{Aut}$(\c)$.
\end{theorem}

\begin{proof}
In view of Proposition \ref{core}, we show that $\c$ is a core. For that  it is sufficient to show $\c$ admits no proper retract (cf. \cite{a.1992Hell}). On contrary, suppose $\c$ admits a proper retract  $\mathcal G'$. Then there exists a homomorphism $f : \c \rightarrow \mathcal G'$ such that $xf = x$ for all $x \in \mathcal G'$. Since the set $ \mathcal E  = \{(1,1), (2,2), (3,3), (4,4)\}$ forms a clique of maximum size as $\omega(\Delta(B_4)) = 4$ (cf. Lemma \ref{clique-4}) so that $\mathcal Ef$ is a clique of size  $4$ (see Remark \ref{r.image-clique-maximum-size}).  By Remark \ref{r.clique}, we have either $\mathcal Ef = \mathcal E$ or $\mathcal Ef = A \times B$  where $A, B \subseteq \{1, 2, 3, 4\}$ with $|A| = |B| = 2$. If $\mathcal Ef =\mathcal E$, then by Proposition \ref{retract-idempotents}, $\mathcal G' = \c$; a contradiction. Thus, $\mathcal Ef = A \times B$. Let $(1, 1)f  = (i, j)$ where $i \ne j$. Then $(i, j)f = (i, j)$ as $(i, j) \in \mathcal G'$. Note that either $i= 1$ or $j  = 1$. If both $i, j \ne 1$, then $(i, j) \sim (1, 1)$. Consequently, $(1, 1)f \sim (i, j)f$ which is not possible as $(i, j)f = (1, 1)f =(i, j)$. Without Loss of generality, we assume that $i = 1$ and $j = 2$. Similarly, $(2, 2)f \in \{(2, k), (k, 2)\}$ for some $k \ne 1, 2$. Since $(2,2)f \sim (1, 2) = (1, 1)f$ as $(1, 1) \sim (2, 2)$.  If $(2, 2)f = (2, k)$, then $(2, k) \sim (1, 2)$; a contradiction of Remark \ref{r.non adjacent vertices} so $(2, 2)f = (k, 2)$ for some $k \ne 1, 2$. Without loss of generality, we suppose $k = 3$. In the same way, we get $(3,3)f = (3,4)$ and $(4,4)f = (1,4)$. Therefore, we have $A = \{1, 3\}$ and $B = \{2,4\}$. In view of Lemma \ref{retraction-image-nonidempotent}, $(2, 4)f \in \{(1, 2), (3, 2), (2, 4)\}$. Since $(1, 1) \sim (2, 4)$ so that $(1,1)f = (1, 2) \sim (2, 4)f$ gives $(2, 4)f = (3, 2)$. 
 Similarly, we get $(2, 3)f = (3, 4)$. Again by  Lemma \ref{retraction-image-nonidempotent}, we have $(1,3)f \in \{(3,2), (3,4), (1,3)\}$. 
For  $(1,3) \sim (2, 3)$ and $(1, 3) \sim (2, 4)$ we obtained $(1,3)f \sim (3, 4)$ and $(1, 3)f \sim (3, 2)$. Consequently, we get a contradiction of Remark \ref{r.non adjacent vertices}.
\end{proof}

\begin{theorem}\label{End}
For $n \ge 5$, we have \rm{End}$(\c)$ = \rm{Aut}$(\c)$.
\end{theorem}

\begin{proof}
In order to prove the result, we show that $\c$ is a core (see Proposition \ref{core}). For that  it is sufficient to show $\c$ admits no proper retract (cf. \cite{a.1992Hell}). On contrary, suppose $\c$ admits a proper retract  $\mathcal G'$. Then there exists an onto homomorphism $f : \c \rightarrow \mathcal G'$ such that $xf = x$ for all $x \in \mathcal G'$. In view of Lemma \ref{retraction-maximum-size-clique}, there exists a clique $K = A \times B$ where $A$ and $B$ forms a partition of $[n]$.  Without loss of generality, we  may assume that $A = \{1, 2, \ldots, t\}$ and $B = \{t+1, t+2, \ldots, n\}$ where $t \in \{\frac{n}{2}, \frac{n - 1}{2}, \frac{n + 1}{2} \}$. Consider the set \[X = \{i \in A \setminus \{1\} : (1, i)f = (1, i) \} \cup \{1 : (2, 1)f = (2, 1)\}.\]

The following claims will be useful in the  sequel. 

\begin{claim}\label{c.i-in-X}
\begin{enumerate}[\rm (i)]
\item For $i \in X$ and $r \ne i \in  A$, we have $(r, i)f  = (r, i)$.
\item For $i \in A \setminus X$ and $r \ne i  \in A$, we have  $(r, i)f = (i, s)$ for some $s \in B$.
\end{enumerate}
\end{claim}

\textit{Proof of Claim} (i) Let $i \ne 1 \in X$. Then $(1, i)f = (1, i)$. If $r \in A \setminus \{1, i\}$, then we have either $(r, i)f = (r, i)$ or $(r, i)f = (i, s)$ where $s \in B$ (cf. Lemma \ref{retraction-image-nonidempotent}). Now, we assume that $(r, i)f = (i, s)$ for some $s \in B$. Since $(r, i)f \sim (1, i)f$ as $(r, i) \sim (1, i)$ so that $(i, s) \sim (1, i)$; a contradiction of Remark \ref{r.non adjacent vertices}. Thus, $(r, i)f = (r, i)$ for all $r \ne i \in A$. Similarly one can observe that if $i = 1 \in X$ and $r \ne i  \in A$, we have $(r, 1)f = (r, 1)$.

\smallskip

\noindent
(ii) First suppose  $i \ne 1 \in A \setminus X$. In view of Lemma \ref{retraction-image-nonidempotent}, we have either $(1, i)f = (1, i)$ or  $(1, i)f = (i, s)$ for some $s \in B$. Note that $(1, i)f \ne (1, i)$ as $i \in A \setminus X$ so $(1, i)f = (i, s)$ for some $s \in B$. If $r \in A \setminus \{1, i\}$, then we have either $(r, i)f = (r, i)$ or $(r, i)f = (i, s')$ where $s' \in B$ (cf. Lemma \ref{retraction-image-nonidempotent}). Suppose $(r, i)f= (r, i)$. Since $(r, i)f \sim (1, i)f$ as $(r, i) \sim (1, i)$ so that $(r, i) \sim (i, s)$; a contradiction of Remark \ref{r.non adjacent vertices}. Thus, $(r, i)f = (i, s')$ for some $s' \in B$.  Similarly, one can observe that if $i = 1 \in A \setminus X$ and $r \ne i  \in A$, we have $(r, 1)f = (1, s)$ for some $s \in B$.

 In view of $X$ we have the following cases.\\
\textbf{Case 1:} Suppose $|X| > |A \setminus X|$. Then  $|X| \geq 2$ as $n \geq 5$. In order to get a contradiction of the fact that $\mathcal G'$ is a proper retract of $\c$, we prove that $f$ is an identity map in this case.  First we show that each non-idempotent element of $\c$ maps to itself under $f$ through the following claim.

\textbf{Note:} If $n > 5$, then $|A| \geq 3$. For $n = 5$, we have either $|A| =2$, $|B| = 3$ or $|A| = 3$, $|B| = 2$. If $|A| =2$ and $|B| = 3$, then $X  = A = \{1, 2\}$.  This case we will  discuss separately in the following claim (vi). Therefore, in part (ii) to (v), we assume that $|A| \geq 3$.
\begin{claim}\label{c.retraction-image-non-idempotent} 

\begin{enumerate}[\rm (i)] 
\item For $p \in A, q\in B$, we have $(p, q)f = (p, q)$.
\item If $p \ne q$ such that $(p, q)f = (a, p)$ for some $a \in A$, then $a \in A \setminus X$.
\item For $p \in B, q\in A$, we have $(p, q)f = (p, q)$.
\item For $p, q\in B$, we have $(p, q)f = (p, q)$.
\item For $p, q \in A$, we have $(p, q)f = (p, q)$.
\item For $n = 5$, $|A| = 2$,  $|B| = 3$ and $p \ne q$, we have $(p, q)f = (p, q)$.
\end{enumerate}
\end{claim}

\textit{Proof of Claim:} (i) Since $K = A \times B$ is contained in $\mathcal G'$ so that $(p, q)f = (p, q)$ for all $p \in A, q\in B$.

\smallskip

\noindent
(ii) On contrary, we assume that $a \in X$. Clearly, $ a \ne p$ (cf. Lemmas \ref{maps-nonidempotent-B_4} and \ref{maps-nonidempotent}). If $p \in A$, then by Claim \ref{c.i-in-X}(i), we get $(p, a)f = (p, a)$. Note that $q \ne a$, otherwise $(p, q)f = (p, q) = (q, p)$  implies $p = q$; a contradiction. Consequently, $(p, q) \sim (p, a)$ gives $(p, a)f = (p, a) \sim (a, p) = (p, q)f$; a contradiction of Remark \ref{r.non adjacent vertices}. Thus, $p \in B$.  For $r  \in A \setminus \{a, q\}$, by Claim \ref{c.i-in-X}(i), we have $(r, a)f = (r, a)$. Since $(p,q) \sim (r, a)$ as $a \ne p$ and $r \ne q$ so that 
$(p,q)f = (a, p) \sim (r, a) = (r, a)f$ which is not possible. Thus, $a \notin X$.

\smallskip

\noindent
(iii) Let $p \in B$ and $q \in A$. First suppose that $q \in X$. Then by Lemma \ref{retraction-image-nonidempotent}, $(p, q)f \in \{(s, p) : s \in A \} \cup \{(q, s) : s \in B \} \cup \{(p, q)\}$. For $r \ne q \in A$, we have $(r, q)f = (r,q)$ (cf. Claim \ref{c.i-in-X}(i)).  Note that $(p, q)f \ne (q, s)$ for any $s \in B$. For instance, if $(p, q)f = (q, s)$ for some $s \in B$, then $(p, q)f = (q, s) \sim (r, q) = (r, q)f$ as $(p, q) \sim (r, q)$, where $r \ne q \in A$; a contradiction of Remark \ref{r.non adjacent vertices}. It follows that $(p, q)f \in \{(s, p) : s \in A \}\cup \{(p, q)\}$. Suppose $(p, q)f = (s, p)$ for some $s \in A$.
Note that $s \in A\setminus X$ (see part (ii)). Now we claim that for any $j \ne q \in X$, we have $(p, j)f = (s', p)$ for some $s' \in A \setminus X$. In view of Lemma \ref{retraction-image-nonidempotent}, $(p, j)f \in \{(s', p) : s \in A \} \cup \{(j, s') : s' \in B \} \cup \{(p, j)\}$. Note that $(p, j)f \ne (p,j)$ because $(p, q) \sim (p, j)$ but $(p, q)f = (s, p)  \nsim (p, j)$ (cf. Remark \ref{r.non adjacent vertices}). In a similar manner, of $(p, q)f \ne (q, s)$ for any $s \in B$, one can show that $(p, j)f \ne (j, s')$ for any $s' \in B$. It follows $(p, j)f = (s', p)$ for some $s' \in A$. By part (ii), we get $(p, j)f = (s', p)$ for some $s' \in A \setminus X$.  Since the subgraph induced by the vertices of the form $(p, j)$ where $j \in X$ forms a clique. Consequently, for any $i \ne j \in X$, we get $(p, i)f = (s, p)$ and $(p, j)f = (s', p)$ are distinct for some $s, s' \in A \setminus X$. Therefore, we have $|X| \leq |A \setminus X|$; a contradiction. Thus, $(p, q)f = (p, q)$ for all $p \in B$ and $q \in X$.

Now we assume $q \in A \setminus X$. In view of Lemma \ref{retraction-image-nonidempotent}, $(p, q)f \in \{(\alpha, p) : \alpha \in A \} \cup \{(q, \beta) : \beta \in B \} \cup \{(p, q)\}$. Suppose $(p, q)f = (\alpha, p)$ for some $\alpha \in A$. In fact $\alpha \in A\setminus X$ (see part (ii)). Choose $i \in X$ as $|X| > |A \setminus X|$, from above we get $(p, i)f = (p, i)$ as $p \in B$. Since $(p, q) \sim (p, i)$ so that $(p, q)f = (\alpha, p) \sim (p, i) = (p, i)f$ which is not possible. Therefore, we have $(p, q)f = (q, \beta)$ for some $\beta \in B$ if $(p, q)f \ne (p, q)$. Again for $i \in X$ and from the above we get $(\beta, i)f = (\beta, i)$. Since  $(p, q) \sim (\beta, i)$ as $p, \beta \in B$ and $ q, i \in A$ gives $(p, q)f = (q, \beta) \sim (\beta, i) = (\beta, i)f$; a contradiction of Remark \ref{r.non adjacent vertices}. Thus,  $(p, q)f = (p, q) \; \forall p \in B$  and $q \in A \setminus X$ and hence the result hold.

\smallskip

\noindent
(iv) Let $p \ne q \in B$. In view of Lemma \ref{retraction-image-nonidempotent}, $(p, q)f \in \{(s, p) : s \in A \}\cup \{(p, q)\}$. Suppose $(p, q)f = (s, p)$ for some $s \in A$. Since $(p, s) \sim (p, q)$ so that $(p, s)f = (p, s) \sim (s, p) = (p,q)f$; a contradiction of Remark \ref{r.non adjacent vertices}. Thus, $(p, q)f = (p, q)$ for all $p, q \in B$.

\smallskip

\noindent
(v) By Claim \ref{c.i-in-X}(i), we have $(p, q)f = (p, q)$ when $q \in X$ so it is sufficient to prove the result for $q \in A\setminus X$. In view of Lemma \ref{retraction-image-nonidempotent}, $(p, q)f \in \{(q, s) : s \in B \} \cup \{p, q)\}$.  Suppose $(p, q)f = (q, s)$ for some $s \in B$. Then by (iv) part, we have  $(s, x)f =(s, x)$ where $x \ne s \in B$. For $p, q \in A$ and $s, x \in B$, we get $(p,q) \sim (s, x)$ gives $(p, q)f = (q, s) \sim (s, x) = (s, x)f$; a contradiction of Remark \ref{r.non adjacent vertices}. Thus, $(p, q)f = (p, q)$ for all $p \ne q \in A$.

%

\smallskip

\noindent
(vi) Suppose $n = 5$, $|A| = 2$,  $|B| = 3$ and $p \ne q$. Then $X = A$ so $(p, q)f = (p, q)$ for all $p, q \in A$ (see Claim \ref{c.i-in-X}(i)). If $p, q \in B$, then by Lemma \ref{retraction-image-nonidempotent}, $(p, q)f \in \{(s, p) : s \in A \} \cup \{p, q)\}$. Suppose $(p, q)f = (s, p)$ for some $s \in A$. Then there exists $s' \in A$ as $|A| = 2$. Consequently, $(s', s)f = (s', s)$ and $(p, q) \sim (s', s)$ gives  $(p, q)f =(s, p) \sim (s', s) = (s', s)f$ which is not possible. Thus, $(p, q)f =(p, q)$ for all $p, q \in B$. Now we suppose that $p \in B$ and $q \in A$. In view of Lemma \ref{retraction-image-nonidempotent}, we have $(p, q)f \in \{(r, p) : r \in A \} \cup \{(q, r') : r' \in B \} \cup \{(p, q)\}$. Suppose $(p, q)f= (r, p)$ for some $r \in A = X$. For $\beta \in B \setminus \{p\}$, we get $(p, q) \sim (p, \beta)$ and $(p, \beta)f = (p, \beta)$ provides  $(s, p) \sim (p, \beta)$ which is not possible. Therefore, $(p, q)f \in \{(q, r') : r' \in B \} \cup \{(p, q)\}$. Let $(p, q)f = (q, r')$ for some  $r' \in B$. Since $|B| = 3$ so that there exists $z \in B \setminus \{p, r'\}$. As a consequence, we have $(r', z) \sim (p, q)$ and $(r', z)f = (r', z)$ implies $(r', z)f = (r', z) \sim (q, r') = (p, q)f$; a contradiction. Thus, $(p, q)f = (p, q)$ for all $p \ne q \in [n]$. 
\vspace{0.5cm}

Thus, by Claim \ref{c.retraction-image-non-idempotent}, we have $(p, q)f = (p, q)$ for all $p \ne q$. Now we show that $(p, p)f = (p, p)$ for all $p \in [n]$. On contrary assume that $(p, p)f = (x, y)$ for some $(x, y) \ne (p,p) \in B_n$. Then $(x, y)f = (x, y)$ as $f$ is a retraction on $\c$. Note that $x \ne y$. Otherwise, $(p, p) \sim (x, y)$ but $(p, p)f = (x, y)f = (x, y)$; a contradiction. Also, observe that $p \in \{x, y\}$. Otherwise, being an adjacent elements $(x, y)$ and $(p, p)$ have same images; again a contradiction. Without loss of generality  assume that  $x= p$. For $z \in [n] \setminus \{y, p\}$, we get $(p, p) \sim (y, z)$ so that  $(p, p)f = (p, y) \sim (y, z) = (y, z)f$; a contradiction of Remark \ref{r.non adjacent vertices}. Thus, $f$ is an identity map. Consequently,  $\mathcal G' = \c$; a contradiction. Thus, \textbf{Case 1}  is not possible.

\textbf{Case 2:} Suppose $|X| \leq |A\setminus X|$. Then $X \ne A$. Now, we have the following subcases depend on $n$.  In each subcase,  we prove that $A = X$ which is a contradiction.

\smallskip

\noindent

\textit{Subcase 1:} $n$ is even. The following claim will be useful in the sequel.
\begin{claim}\label{c.retraction-A_i-B_{t_i}}
\begin{enumerate}[\rm (i)]
\item Let $i \in A\setminus X$. Then there exists a unique $s_i \in B$ such that the restriction map  $f_{A_i \times B_{s_i}}$ of $f$ is a bijection from $A_i = \{(r, i) : r \ne i \in A\} \; \text{onto} \; B_{s_i} = \{(i, s) : s \ne s_i \in B\}$. 

\item In view of part (i), for $Y = \{s_i \in B : i \in A \setminus X\}$, we have $Y = B$. Moreover, for $i \ne j \in A \setminus X$, we have $s_i \ne s_j$.
\item If $x \ne y \in B$, then $(x, y)f = (x, y)$.
\item If $i \ne j \in A$, then $(i, j)f = (i, j)$. 
\end{enumerate}
 \end{claim}

\textit{Proof of Claim:} (i) Let $i \in A\setminus X$. Then for $r \ne i \in A$, we have  $(r, i)f =(i, s)$ for some  $s \in B$ (see Claim \ref{c.i-in-X}(ii)). Consequently, $A_if \subset \{(i, s) : s \in B \}$.  Since $f$ is one-one on $A_i$ because $A_i$ forms a clique, we get $|A_if| = |A_i| = |A| - 1 = |B| - 1$ as $n$ is even. Thus, there exists $s_i \in B$ such that $A_if = B_{s_i}$, where $B_{s_i} = \{(i, s) : s \in B \setminus \{s_i\}\}$.
Hence,  $f_{A_i \times B_{s_i}}$ is a one-one map from $A_i$ onto $B_{s_i}$.

\smallskip

\noindent
(ii) Clearly $Y \subseteq B$. We show that $Y \subset B$ is not possible. On contrary, if $Y \subset B$ so there exists $s \in B \setminus Y$. Let $x \ne s \in B$. By Lemma \ref{retraction-image-nonidempotent}, $(s, x)f \in \{(\alpha, s) : \alpha \in A \}\cup \{(s, x)\}$. We provide a contradiction for both the possibilities of $(s, x)f$. Suppose $(s, x)f = (\alpha, s)$ for some $ \alpha \in A$. By Claim \ref{c.retraction-image-non-idempotent}(ii), in fact we have $(s,x)f = (\alpha, s)$ for some $\alpha \in A \setminus X$. Then by part (i) there exists $s_{\alpha} \in B$ such that the map  $f_{A_{\alpha} \times B_{s_{\alpha}}}$ is a bijection. As $s_{\alpha} \in Y$, $s \ne s_{\alpha}$ so that $(\alpha, s) \in B_{s_{\alpha}}$. Consequently, there exists $r_{\alpha} \ne \alpha \in A$ such that $(r_{\alpha}, \alpha)f = (\alpha, s)$. Now since $r_{\alpha}, \alpha \in A$ and $s, x \in B$ we get $(r_{\alpha}, \alpha) \sim (s, x)$ as $A$ and $B$ forms a partition of $[n]$ so that  $(r_{\alpha}, \alpha)f \sim (s, x)f$. But $(r_{\alpha}, \alpha)f = (s, x)f = (\alpha, s)$ which is not possible. It follows that $(s, x)f = (s, x)$. For $i \in A \setminus X$, there exists $s_i \in Y$ such that the map $f_{A_{i} \times B_{s_{i}}}$ is a bijection. Since $s \ne s_i$ as $s \notin Y$ gives $(i, s) \in B_{s_i}$. As a result, there exists $r \ne i \in A$ such that $(r, i)f = (i, s)$. For $r, i \in A$ and $s, x \in B$, we get $(s, x) \sim (r, i)$; again a contradiction as $(s, x)f = (s, x) \sim (i, s) = (r_{i},i)f$. Hence, $Y = B$. 

\smallskip

\noindent
(iii)  Let $x, y \in B$. Then by Lemma \ref{retraction-image-nonidempotent}, $(x, y)f \in \{(\alpha, x) : \alpha \in A \} \cup \{(x, y)\}$. Suppose $(x, y)f = (\alpha,x)$ for some $\alpha \in A$. In fact $\alpha \in A \setminus X$ ( see Claim \ref{c.retraction-image-non-idempotent}(ii)). 
For $x \in B = Y$, there exists $i_x \in A \setminus X$ such that  $f_{A_{i_x} \times B_x }$ is a bijection. If $\alpha \ne i_x \in A \setminus X$, then by part (i) there exists $s  _{\alpha} \in B \setminus \{x\}$ such that the  restriction map $f_{A_{\alpha} \times B_{s_{\alpha} }}$ is a bijective map and $(\alpha, x) \in B_{s_{\alpha}}$. Consequently, we get $(r, \alpha)f = (\alpha, x)$ for some $r \ne \alpha \in A$. But $(x, y) \sim (r, \alpha)$ as $x, y \in B$ and $r, \alpha \in A$ gives $(x, y)f \ne (r, \alpha)f$. However, we have $(x, y)f = (r, \alpha)f$; a contradiction. It follows that $\alpha = i_x$. In view of Lemma \ref{retraction-image-nonidempotent}, for $y' \in B\setminus \{x, y\}$, note that $(x, y')f \in \{(\alpha', x) : \alpha' \in A \} \cup \{(x, y')\}$. Now observe that $(x, y')f \ne (x, y')$. If $(x, y')f = (x, y')$, then $(x, y) \sim (x, y')$ provides $(\alpha, x) \sim (x, y')$; a contradiction of Remark \ref{r.non adjacent vertices}. Thus, $(x, y')f = (\alpha', x)$ for some $\alpha' \in A \setminus X$. Further note that $\alpha' \ne \alpha$. Otherwise, $(x, y) \sim (x, y')$ gives $(x, y)f \sim (x, y')$ but $(x, y)f = (x, y')f = (\alpha, x)$ which is not possible. Consequently, $\alpha' \ne i_x$. By the similar argument used for $\alpha \ne i_x$, we get  $(r', \alpha')f = (\alpha', x)$ for some $r' \ne \alpha' \in A$. Since $(r', \alpha') \sim (x, y')$ we get $(r', \alpha')f \sim (x, y')f$ but  $(r', \alpha')f =  (x, y')f = (\alpha', x)$ is not possible. Hence, $(x, y)f = (x, y)$ for all $x\ne y \in B$.

\smallskip

\noindent
(iv) Suppose $i \ne j \in A$. Then by Lemma \ref{retraction-image-nonidempotent}, $(i, j)f \in \{(j, \beta) : \beta \in B \} \cup \{(i, j)\}$. If $(i, j)f = (j, \beta)$ for some $\beta  \in B$ then for $x \in B\setminus \{\beta \}$  note that $(i, j)\sim (\beta, x)$ but    $(i, j)f = (j, \beta) \nsim (\beta, x) = (\beta, x)f$ (cf. part (iii)). Thus, $(i, j)f = (i, j)$.

\vspace{0.3cm}

By Claim \ref{c.retraction-A_i-B_{t_i}}(iv), we get  $A = X$. Therefore, \textbf{Case 2} is not possible when $n$ is even. 

\vspace{0.3cm}
\textit{Subcase 2:} $n$ is odd. By Lemma \ref{retraction-maximum-size-clique}, we have either $|A| = \frac{n + 1}{2}$, $|B| = \frac{n - 1}{2}$ or $|A| = \frac{n - 1}{2}$, $|B| = \frac{n + 1}{2}$ (see proof of Lemma \ref{clique inequality}). First we prove the following claim.
\begin{claim}\label{c.case-2-odd}
\begin{enumerate}[\rm (i)]
\item If $x \ne y \in B$, then $(x, y)f = (x, y)$.
\item If $x \in B$ and $i \in A$, then $(x, i)f = (x, i)$.
\end{enumerate}
\end{claim}
\textit{Proof of Claim:} (i) First, we suppose that $|A| = \frac{n + 1}{2}$ and $|B| = \frac{n - 1}{2}$.  Let  $x \ne y \in B$. Then by Lemma \ref{retraction-image-nonidempotent}, we get either $(x, y)f = (i, x)$ for some $i \in A$ or $(x, y)f = (x, y)$. Let if possible, $(x, y)f  = (i, x)$ for some $i \in A$. In fact $i \in A \setminus X$ (cf. Claim \ref{c.retraction-image-non-idempotent}(ii)). Also, for $r \ne i \in A$ and $i \in A \setminus X$, by Claim \ref{c.i-in-X}(ii), we get $(r, i)f = (i, s)$ for some $s \in B$. As a result, $A_if \subseteq B_i$ where $A_i = \{(r, i) : r\ne i \in A \}$ and $B_i = \{(i, s): s \in B \}$. Since $A_i$ forms a clique, we have $f$ is one-one on $A_i$. Moreover, $|A_if| = |A_i| = |A| - 1 = |B| = |B_i|$. Therefore, we get a bijection  $f_{A_i \times B_i}$ from $A_i$ onto $B_i$. Then there exists $r \ne i \in A$ such that $(r, i)f = (i, x)$ for some $x \in B$. Note that $(x, y) \sim (r, i)$ but $(x, y)f = (r, i)f = (i, x)$ which is not possible. Thus, $(x, y)f  = (x, y)$ for all $x \ne y \in B$. 

On the other hand, we may assume that $|A| = \frac{n - 1}{2}$ and $|B| = \frac{n + 1}{2}$. Then $|B| \geq 3$. First, we claim that there exist $x, y \ne B$ such that $(x, y)f = (x, y)$. On contrary, we assume that $(x, y)f \ne (x, y)$ for all $x \ne y$ in $B$. Let $x \ne y \in B$. By Lemma \ref{retraction-image-nonidempotent} and Claim \ref{c.retraction-image-non-idempotent}(ii), we have $(x, y)f = (\alpha, x)$ for some   $\alpha \in A \setminus X$. Similarly, for any $y' \in B \setminus \{x, y\}$, we have $(x, y')f = (\alpha', x)$ for some $ \alpha' \in A \setminus X$.  It follows that $B_xf \subseteq A_x$ where $B_x = \{(x, z) : z \ne x \in B\}$ and $A_x = \{(i, x) : i \in A\setminus X\}$. Since the set $B_x$ forms a clique so that $f$ is one-one on $B_x$ provide $|B_xf| = |B_x| =   |B|-1 = |A| = |A_x| = |A\setminus X|$. Consequently, we get $f_{B_x \times A_x}$ is a bijection and $X = \varnothing$. For $r \ne \alpha \in A$, we have $(r, \alpha)f = (\alpha, \beta)$ for some $\beta \in B$ (cf. Claim \ref{c.i-in-X}(i)). If $\beta = x$, then $(x, y)f = (r, \alpha)f = (\alpha, x)$ but $(x, y) \sim (r, \alpha)$ which is not possible. For $\beta \ne x$, by using the similar argument used for $x$, there exist the subsets $B_{\beta}$ and $A_{\beta}$ such that the restriction map $f_{B_{\beta} \times A_{\beta}}$ is a bijective map. As a consequence $(\alpha, \beta) \in A_{\beta}$ so that there exists $(\beta, s) \in B_{\beta}$ such that $(\beta, s)f = (\alpha, \beta)$. As $ r, \alpha \in A$ and $\beta, s \in B$, $( r, \alpha) \sim (\beta, s)$ gives $( r, \alpha)f \sim (\beta, s)f$ but $( r, \alpha)f = (\beta, s)f = (\alpha, \beta)$ which is not possible. Thus, there exist $p \ne q \in B$ such that $(p, q)f = (p, q)$. 

For any $w \in B \setminus \{p, q\}$, we have either $(p, w)f = (p, w)$ or $(p, w)f = (i, p)$ for some $ i \in A$. Since $(p, q) \sim (p, w)$ so that $(p, q)f = (p, q) \sim (p, w)f$ implies $(p, w)f \ne (i, p)$ for any $i \in A$. Therefore, $(p, w)f = (p, w)$. Consider the subsets $A' = A \cup \{p\}$ and $B' = B \setminus \{p\}$ of $[n]$. Note that $A'$ and $B'$ are the disjoint subsets of $[n]$ with $|A'| = \frac{n + 1}{2}$ and $|B'| = \frac{n - 1}{2}$ so $A' \times B'$ forms a clique of maximum size in $\mathcal G'$. If $|X| > |A' \setminus X|$, then in Claim \ref{c.retraction-image-non-idempotent}(iv), replace $A$ and $B$ with $A'$ and $B'$ respectively, we get $(a, b)f = (a, b)$ for all $a, b \in B'$. For $|X| \leq |A' \setminus X|$, by using the similar concept used above we have $(a, b)f = (a, b)$ for all $a, b \in B'$. Since $(p, w)f = (p, w)$ for all $w \ne x \in B$ so that 
$(a, b)f = (a, b)$ for all $a, b \in B$ and $b \ne x$. If possible, let $(a, p)f \ne (a,p)$, then by Lemma \ref{retraction-image-nonidempotent}, $(a, p)f = (l, a)$ for some $l \in A$. Choose $\beta \in B  \setminus \{a, p\}$ so $(a, \beta) \sim (a, p)$ and $(a, \beta)f = (a, \beta)$ as $a, \beta \in B'$ we obtained $(a, \beta)f = (a, \beta) \sim (l, a) = (a, p)$; a contradiction of remark \ref{r.non adjacent vertices}. Hence, $(a, b)f = (a, b)$ for all $a, b \in B$.

\smallskip

\noindent
(ii) Let $x \in B$ and $i \in A$. Then by Lemma \ref{retraction-image-nonidempotent}, we have $(x, i)f \in \{ (\alpha, x) : \alpha \in A\} \cup \{ (i, \beta) : \beta \in B \} \cup \{(x, i)\}$. Note $(x, i)f \ne (\alpha, x)$ for any $\alpha \in A$. For instance if $(x, i)f = (\alpha, x)$ for some $\alpha \in A$, then $(x, y) \sim (x, i)$ where $y \ne x \in B$ gives $(x, y)f \sim (x, i)f$. By part (i), we get $(x, y)f = (x, y)$ so $(x, y) \sim (\alpha, x)$; a contradiction of Remark \ref{r.non adjacent vertices}. On the other hand now we get a contradiction for  $(x, i)f = (i, \beta)$ for some $\beta \in B$. If $\beta = x$ then for $\gamma \ne x \in B$, we have $(x, \gamma)f = (x, \gamma)$ ( by part (i)). Since $(x, i) \sim (x, \gamma)$ but $(x, i)f =(i, x) \nsim (x, \gamma) = (x, \gamma)f$ which is not possible so $\beta \ne x$. For $n \geq 5$, we have $|B| \geq 2$. If  $|B| = 2$, then $|A| = 3$. There exists $j, k \in A \setminus \{i\}$. Consequently, $(j, i)f = (i, y)$ and $(k, i)f = (i, z)$ for some $y, z \in B$. Because if $(j, i)f = (j, i)$ (cf. Lemma \ref{retraction-image-nonidempotent}) then $(x, i) \sim (j, i)$ gives $(x, i)f = (i, \beta) \sim (j, i) = (j, i)f$; a contradiction of Remark \ref{r.non adjacent vertices}. Similarly, $(k, i)f = (k, i)$ is not possible. Note that $\{(x, i), (j, i), (k, i)\}$ forms a clique of size $3$ so that $\{(x, i)f, (j, i)f, (k, i)f\} =$ \\ $\{(i, y), (i, z), (i, s)\}$. Consequently, $\beta, y, z$ are the elements of $B$. Thus,  $|B| \geq 3$; a contradiction of $|B| = 2$. It follows that $|B| \geq 3$. For $z \in B \setminus \{x, s\}$ we have $(x, i) \sim (\beta, z)$. By part (i), $(\beta, z)f = (\beta, z)$. Consequently, $(x, i)f = (i, \beta) \sim (\beta, z) = (\beta, z)f$ which is not possible. Hence, $(x, i)f = (x, i)$.

\vspace{0.5cm}

Now if $x \in A$, then $i \in A \setminus X$. For $x \in B$, by Claim \ref{c.case-2-odd}(ii), we have $(x, i)f = (x, i)$. Since $(1, i) \sim (x, i)$ so that $(1, i)f = (i, s) \sim (x, i) = (x, i)f$; a contradiction of Remark \ref{r.non adjacent vertices}. Thus, $X \subset A$ is not possible. Consequently, $X = A$; a contradiction of \textbf{Case 2}. In view of \textbf{Case 1} and \textbf{Case 2} such $X$ is not possible. Thus, $\c$ admits no proper retract. Hence, $\c$ is a core.
\end{proof}

\textbf{Open Problem:} The work in this paper can be carried out for other class of semigroups viz. the semigroup of all partial maps on a finite set and its various subsemigroups. In view of Theorem \ref{0-simple}; to investigate  the commuting graph of finite $0$-simple inverse semigroup, it is sufficient to investigate $\Delta(B_n(G))$. In this connection, the results obtained in this paper might be useful. For example, using the result of $\c$, in particular Theorem \ref{Hamiltonian}(iii), we prove the following theorem which gives a partial answer to the problem posed in \cite[Section 6]{a.Araujo2011}. 

\begin{theorem}
For $n \ge 3$, $\Delta(B_n(G))$ is Hamiltonian.
\end{theorem}

\begin{proof} Let $G = \{a_1, a_2, \ldots, a_m\}$. We show that there exists a Hamiltonian cycle in $\Delta(B_n(G))$. First note that if $(i, j) \sim (k, l)$ in $\c$, then $(i, a, j) \sim (k, b, l)$ in $\Delta(B_n(G))$ for all $a, b \in G$. Let $G_{a_1} = \{(i, a_1, j) \; : \; i, j \in [n] \}$. Since $\c$ is Hamiltonian (see Theorem \ref{Hamiltonian}(iii)), we assume that there exists a Hamiltonian cycle $C$ . Corresponding to the cycle $C$, choose a Hamiltonian path $P$ whose first vertex is $(i, j)$ and the end vertex is $(k, l)$. For the path $P$, there exists a Hamiltonian path in the subgraph induced by $G_{a_1}$ whose first vertex is $(i, a_1, j)$ and the end vertex is $(k, a_1, l)$. Since $(i, j)  \sim (k, l)$ in $\c$, we have $(k, a_1, l) \sim (i, a_2, j)$. By the similar way, we get a Hamiltonian path in the subgraph induced by $G_{a_2}$ whose first vertex is $(i, a_{2}, j)$ and the end vertex is $(k, a_2, l)$. On Continuing this process, we get a Hamiltonian path in $\Delta(B_n(G))$ with first vertex is $(i, a_1, j)$ and the end vertex is $(k, a_m, l)$. For $(i, j) \sim (k, l)$, we get $(i, a_1, j) \sim (k, a_m, l)$. Thus, $\Delta(B_n(G))$ is Hamiltonian.
\end{proof}

\noindent\textbf{Acknowledgement:}
The first author wishes to acknowledge the support of MATRICS Grant \break (MTR/2018/000779) funded by SERB, India.

\end{document}